\documentclass{amsproc}
\usepackage{amssymb}
\usepackage{amsmath}
\usepackage{amsfonts}

\theoremstyle{plain}

\newtheorem{definition}{Definition}

\newtheorem{lemma}{Lemma}

\newtheorem{remark}{Remark}

\newtheorem{theorem}{Theorem}
\numberwithin{equation}{section}
\input{tcilatex}

\begin{document}
\title[Speedups of compact group extensions]{Speedups of compact group
extensions}
\author{Andrey Babichev}
\address{Department of Mathematics, California State University, Fresno, CA
93740}
\email{ababichev@csufresno.edu}
\urladdr{}
\thanks{ }
\author{Adam Fieldsteel}
\address{Department of Mathematics and Computer Science, Wesleyan
University, Middletown, CT 06459}
\email{afieldsteel@wesleyan.edu}
\urladdr{}
\thanks{ }
\date{May 8, 2008}

\begin{abstract}
Let $S_{1}$ and $S_{2}$ be ergodic extensions of finite measure preserving
transformations $T_{1}$ and $T_{2},$ where the extensions are by rotations
of a compact group $G.$ Then there is an $\mathbb{N}-$valued function $k,$
measurable with respect to the factor $T_{1},$ so that $S_{1}^{k}$ is
isomorphic to $S_{2}$ by an isomorphism that respects the action of $G$ on
fibers.
\end{abstract}

\maketitle

\section{Introduction}

Fix a compact group $G$ with Haar measure $\lambda $ and two-sided invariant
metric $\rho \leq 1.$ Let $T$ be a measure preserving transformation of the
Lebesgue probability space $\left( X,\mathcal{A},\mu \right) $ and $\sigma
:X\rightarrow G$ an $\mathcal{A}-$measurable map. The transformation $%
S:X\times G\rightarrow X\times G$ given by $S\left( x,g\right) =\left(
Tx,\sigma \left( x\right) g\right) $ is a measurable map preserving $\mu
\times \lambda .$ We refer to such an $S$ as a $G-$extension of $T,$ or more
briefly as a $G-$extension, if $T$ is either understood or need not be
specified. The factor $T$ will be referred to as the base factor of $S$ and
we will frequently identify the sets in $\mathcal{A}$ with their preimages
in $X\times G$ under the projection on the first coordinate$.$ We will use
the notation $\left( S,T,X,\sigma \right) $ to denote such a $G-$extension,
and we will use abbreviations such as $\left( T,\sigma \right) $ or $S$ when
the other components are understood. We will adopt the notational convention
that all $G-$extensions are represented by the letter $S,$ or a modified
letter $S,$ and the associated base factor and function into $G\,$will be
represented by the letters $T,X,$ and $\sigma ,$ respectively, with the same
modifiers. Thus a $G-$extention $S^{\prime }$ is understood to be associated
with the components $\left( T^{\prime },X^{\prime },\mathcal{A}^{\prime
},\mu ^{\prime },\sigma ^{\prime }\right) .$ We let $c$ denote the
projection $c:\left( x,g\right) \mapsto g.$

Each $G-$extension admits a natural free action of $G$ on $X\times G$,
which, for each $h\in G,$ is given by 
\begin{equation*}
h\left( x,g\right) =\left( x,gh\right)
\end{equation*}%
and this action commutes with the action of $\mathbb{Z}$ given by (the
powers of) $S.$ For each $\left( x,g\right) \in X\times G,$ we refer to the
set $G\left( x,g\right) =\left\{ \left( x,gh\right) \mid h\in G\right\} $ as
the $G-$orbit or the $G-$fiber of $\left( x,g\right) .$

Given two $G$ extensions $\left( S,T,X,\sigma \right) $ and $\left( \bar{S},%
\bar{T},\bar{X},\bar{\sigma}\right) $ we say $S$ is a $G-$factor of $\bar{S}$
if there is a factor map $\Phi $ from $\bar{S}$ to $S$ of the form%
\begin{equation*}
\Phi \left( \bar{x},g\right) =\left( \phi \left( \bar{x}\right) ,\bar{\alpha}%
\left( \bar{x}\right) g\right)
\end{equation*}%
where $\phi $ is a factor map from $\bar{T}$ to $T$ and $\bar{\alpha}:\bar{X}%
\rightarrow G$ is an $\mathcal{\bar{A}}-$ measurable function. If such a $%
\Phi $ exists for which $\phi $ is an isomorphism from $\bar{T}$ to $T,$ we
say $S$ is $G-$isomorphic to $\bar{S}.$ We note that these relations can be
described in terms of cocycles on equivalence relations. We will not make
use of this language, so we omit the definitions, but we simply state: given
a $G-$extension $\left( S,T,X,\sigma \right) ,$ the function $\sigma $
determines (and is determined by) a $G-$valued cocycle on the orbit relation
of $T.$ The condition that two $G-$extensions are $G-$isomorphic says that
the base transformations are isomorphic, and after this identification of
the orbit relations of the base transformations, their associated $G-$%
cocycles are cohomologous. The function $\bar{\alpha}$ is the
\textquotedblleft transfer function\textquotedblright\ that relates the two
cocycles.

By a speedup of a transformation $T:X\rightarrow X$ we mean a transformation 
$T^{\prime }:X\rightarrow X$ of the form $T^{\prime }\left( x\right)
=T^{k\left( x\right) }\left( x\right) ,$ for some measurable $k:X\rightarrow 
\mathbb{N}$. Given a $G-$extension $\left( S,T,X,\sigma \right) $ we
consider speedups of $S$ for which the variable exponent $k$ is measurable
with respect to the base factor, and we refer to such a transformation as a $%
G-$speedup of $S$. Each $G-$speedup of $S$ determines, and is determined by,
a speedup of the base factor $T.$ Thus a $G-$speedup of $S$ can be
understood to be a $G-$extension $S^{\prime }$ of the form 
\begin{equation*}
S^{\prime }\left( x,g\right) =\left( T^{\prime }x,\sigma ^{\prime }\left(
x\right) g\right) ,
\end{equation*}%
where 
\begin{equation*}
T^{\prime }\left( x\right) =T^{k\left( x\right) }\left( x\right)
\end{equation*}%
for an $\mathcal{A}-$measurable function and $k:X\rightarrow \mathbb{N}$,
and 
\begin{equation}
\sigma ^{\prime }\left( x\right) =\sigma ^{\left( k\right) }\left( x\right)
=\sigma \left( T^{k\left( x\right) -1}\left( x\right) \right) ...\sigma
\left( Tx\right) \sigma \left( x\right) .  \label{sigma cocycle}
\end{equation}

Our goal here is to prove that for all ergodic $G-$extensions $S$ and $\bar{S%
},$ $S$ can be obtained as a $G-$speedup of $\bar{S}$. That is, there is a $%
G-$speedup of $\bar{S}$ that is $G-$isomorphic to $S.$ We note that the
restriction of this theorem to the special case where the group $G$ is
trivial is a result obtained by Arnoux, Ornstein and Weiss \cite{AOW}, and
our work here gives a new proof of that result.

The theorem is an analogue of the orbit equivalence result for $G-$%
extensions obtained in \cite{F} and independently by other methods in \cite%
{G}. The proof will fall into two main parts. First we will show that, given
such $S$ and $\bar{S},$ there is an ergodic $G-$speedup of $\bar{S}$ that
has $S$ as a $G-$factor. We will then improve this result to obtain an
isomorphism. From a broader point of view, the overall argument is carried
out by an argument that is closely related to those of the theory of
restricted orbit equivalence developed by Rudolph and Kammeyer \cite{R},\cite%
{KR1},\cite{KR2}, and that is ultimately derived from Ornstein's proof of
the isomorphism theorem for Bernoulli shifts \cite{O}.

The general idea of the proof is a natural one, which\ may be obscured by
its implementation. Briefly, to obtain a speedup of a transformation $\bar{T}
$ that is isomorphic to a transformation $T$ we must advance along $\bar{T}-$%
orbits so that, with respect to a suitable partition $\bar{P},$ we visit the
elements of $\bar{P}$ in a manner that imitates the behavior of the orbits
of $T$ with respect to a generating partition $P.$ The ergodicity of $\bar{T}
$ will make this possible. To obtain a $G-$speedup of a $G-$extension $\bar{S%
}$ that is isomorphic to a given $S,$ we do the same, with the additional
requirement that we advance along $\bar{S}-$orbits by amounts that are
constant on $G-$fibers, in a manner that imitates the behavior of the orbits
of $S,$ with respect to both the first and second coordinates.

A particular technical issue that will concern us here, which was not
present in the earlier work on orbit equivalence \cite{F}$,$ is that of
establishing the ergodicity of our speedups. Ergodicity is preserved under
orbit equivalence, but the orbits of a speedup are suborbits of an ergodic
transformation, so special effort will be needed to ensure that the speedups
we construct are ergodic. To simplify matters a bit, the main argument will
be carried out first in the case of finite partitions and then extended to
allow countable partitions.

\section{Preliminaries}

\subsection{Partial transformations}

The speedups of the theorem will be obtained as limits of partially defined
transformations, which we now introduce.

\begin{definition}
A \emph{partial transformation} $T\ $on $X$ is an injective,
measure-preserving map $T:Dom\left( T\right) \rightarrow X$ defined on a
measurable subset $Dom\left( T\right) $ of $X.$ For such $T$ and for $n\in 
\mathbb{Z}$ we obtain a partial transformation $T^{n}$ in a natural way. For
each set $C\subset \mathbb{Z}$ and $x\in X$ we let $T^{C}x=\left\{
T^{n}x\mid n\in C\text{ and }x\in Dom\left( T^{n}\right) \right\} .$ In
particular, we refer to $T^{\mathbb{Z}}x$ as the $T-$orbit of $x.$ A \emph{%
partial }$G-$\emph{extension} (of a partial transformation $T$) is a map $%
S:Dom\left( T\right) \times G\rightarrow X\times G$ of the form%
\begin{equation*}
S\left( x,g\right) =\left( Tx,\sigma \left( x\right) g\right)
\end{equation*}%
where $\sigma :Dom\left( T\right) \rightarrow G$ is $\mathcal{A}-$measurable.
\end{definition}

\begin{definition}
A \emph{partial speedup} of a transformation $T_{0}:X\rightarrow X$ is a
partial transformation $T:Dom\left( T\right) \rightarrow X$ that satisfies%
\begin{equation*}
T\left( x\right) =T_{0}^{k\left( x\right) }\left( x\right)
\end{equation*}%
for all $x\in Dom\left( T\right) ,$ where $k:Dom\left( T\right) \rightarrow 
\mathbb{N}$ is a measurable function. If $\left( S_{0},T_{0},X,\sigma
\right) $ is a $G-$extension then a \emph{partial }$G-$\emph{speedup} of $%
S_{0}$ is a partial speedup $S$ of $S_{0}$ where the domain of $S\ $is a
measurable set of the form $Dom\left( S\right) =X^{\prime }\times G,$ and $S$
has the form 
\begin{equation*}
S\left( x,g\right) =S_{0}^{k\left( x\right) }\left( x,g\right)
\end{equation*}%
where $k:X^{\prime }\rightarrow \mathbb{N}$ is\ an $\mathcal{A}-$measurable
function. Equivalently, we can view $S$ as the partial $G-$extension of the
partial speedup $T$ of $T_{0}$ with domain $X^{\prime }$, where $T$ is given
by the same exponent $k,$ and $T$ is extended by the function $\sigma
^{\left( k\right) }$ as in $\left[ \ref{sigma cocycle}\right] .$\newline
\end{definition}

\begin{definition}
If $\left( S,T,\sigma ,X\right) $ is a $G-$extension and $\alpha
:X\rightarrow G$ is measurable then we let $\left( S^{\alpha },T,\sigma
^{\alpha },X\right) $ denote the $G-$extension given by setting 
\begin{equation*}
\sigma ^{\alpha }\left( x\right) =\alpha \left( Tx\right) \sigma \left(
x\right) \alpha ^{-1}\left( x\right)
\end{equation*}%
We note that $\left( S^{\alpha },T,\sigma ^{\alpha },X\right) $ is $G-$%
isomorphic to $\left( S,T,\sigma ,X\right) $ via the isomorphism 
\begin{equation*}
\left( x,g\right) \rightarrow \left( x,\alpha \left( x\right) g\right) .
\end{equation*}%
A similar definition is made, using the same notation, in the case that $%
\left( S,T,\sigma ,X\right) $ is a partial $G-$extension, with $\alpha
:Dom\left( T\right) \rightarrow G.$
\end{definition}

\subsection{Distributions and Sampling}

Let $\mathcal{M}\left( M\right) $ denote the space of Borel probability
measures on the metric space $\left( M,\rho \right) ,$ where $\left( M,\rho
\right) $ is taken to be separable, and $\rho $ is bounded by $1.$ These
conditions on $M$ will be understood to be in effect throughout this paper.
We make use of the Kantorovich metric on $\mathcal{M}\left( M\right) $
(which yields the weak topology on $\mathcal{M}\left( M\right) $) defined by
setting, for all $\lambda _{1},\lambda _{2}\in \mathcal{M}\left( M\right) ,$%
\begin{equation*}
\left\Vert \lambda _{1},\lambda _{2}\right\Vert _{\mathcal{M}}=\inf \left\{
\int_{M\times M}\rho \left( x,y\right) d\nu \left( x,y\right) \right\}
\end{equation*}%
where the infimum is taken over all probability measures $\nu $ on $M\times
M $ having marginals $\lambda _{1}$ and $\lambda _{2}.$

Given a (Borel) measurable function $f$ from a Lebesgue space $\left( X,\mu
\right) $ to $\left( M,\rho \right) ,$ by the \emph{distribution} of $f$,
denoted $dist_{X}\left( f\right) $ we mean the image of $\mu $ under $f.$ If 
$Y$ is a subset of $X$ with $\mu \left( Y\right) >0,$ we obtain $%
dist_{Y}\left( f\right) $ by restricting $f$ to the normalized measure space 
$\left( Y,\frac{\mu }{\mu \left( Y\right) }\right) .$ If $A$ is a finite
subset of $X,$ we obtain $dist_{A}\left( f\right) $ by restricting $f$ to
the space $\left( A,\nu \right) ,$ where $\nu $ is normalized counting
measure on $A.$ When $M$ is a finite or countable set and no other metric is
specified, it will be understood that $\rho $ is the discrete metric on $M$.

For $K\in \mathbb{N}$ we let $\left[ K\right] $ denote the set $\left\{
0,1,...,K-1\right\} .$ Given $n$ and $K\in \mathbb{N}$ and a sequence $s:%
\left[ K\right] \rightarrow M,$ we obtain a function $s_{n}:\left[ K-n+1%
\right] \rightarrow \left( M^{n}\right) ^{K-n+1}$ by setting, for each $i\in %
\left[ K-n+1\right] ,$ 
\begin{equation*}
s_{n}\left( i\right) =\left( s\left( i\right) ,...,s\left( i+n-1\right)
\right) .
\end{equation*}%
We refer to $dist_{\left[ K-n+1\right] }\left( s_{n}\right) $ as the $n-$%
distribution of $s.$ The restrictions of $s$ to intervals of length $n$ are
called $n-$blocks in $s.$ More generally, if we specify a set of $n-$blocks
in $s,$ where $I\subset \left[ K-n+1\right] $ is the set of initial
positions of these blocks, then we refer to $dist_{I}\left( s_{n}\right) $
as the $n-$distribution of this set of $n-$blocks. (We will make use of this
especially in the case where the specified $n-$blocks (that is, their
domains) are pairwise disjoint). The sequences to which we apply this
language will often be the values of a function $f$ along an orbit of a
transformation $T$. In that case the sequence $\left\{ f\left( T^{i}x\right)
\right\} _{i=0}^{n}$ will be referred to as the $T-f-n-$name\ of $x$. We
will also use this language in connection with orbits themselves. In
particular, if $S$ is a speedup of $S_{0},$ we may need to speak about
blocks in $S-$orbits as well as blocks in $S_{0}-$orbits, so to distinguish
them, we will refer to $S-$blocks and $S_{0}-$blocks.

The Birkhoff ergodic theorem can be formulated as:\bigskip

\textsc{Ergodic Theorem: }\emph{Let }$T$\emph{\ be an ergodic measure
preserving transformation of }$\left( X,\mu \right) $\emph{, and }$%
f:X\rightarrow \left( M,\rho \right) $\emph{\ a measurable function. Then
for almost every }$x\in X,$%
\begin{equation*}
\lim_{n\rightarrow \infty }\left\Vert dist_{T^{\left[ n\right] }\left(
x\right) }\left( f\right) ,dist_{X}\left( f\right) \right\Vert _{\mathcal{M}%
}=0.
\end{equation*}%
\bigskip

When 
\begin{equation*}
\left\Vert dist_{T^{\left[ n\right] }\left( x\right) }\left( f\right)
,dist_{X}\left( f\right) \right\Vert _{\mathcal{M}}<\zeta
\end{equation*}%
we say that the $T-f-n-$name of $x$ has $\zeta -$good distribution. If the
function $f$ has the form 
\begin{equation*}
f=\tbigvee\limits_{i\in \left[ k\right] }T^{-i}g
\end{equation*}%
then in the above situation we would say that the $T-g-n-$name of $x$ has $%
\zeta -$good $k-$distribution.

The following two lemmas provide key combinatorial devices that will be used
in our argument.

\begin{lemma}
\label{model name 1}Let $\left( T,X,\mu \right) $ be an ergodic
transformation and $f:X\rightarrow \left( M,\rho \right) $ a measurable
function$.$ For all $n\in \mathbb{N}$ and $\zeta >0$ there exists $L\left(
n,\zeta \right) \in \mathbb{N}$ so that for all $L\geq L\left( n,\zeta
\right) ,$ $\left( 1-\zeta \right) -$most points have $L-$names that can be $%
\left( 1-\zeta \right) -$covered by a set of disjoint $n-$blocks which has $%
\zeta -$good $n-$distribution. In addition, these $n-$blocks are organized
into groups of consecutive $n-$blocks where the concatenation of these
groups has $\zeta -$good $n-$distribution. Moreover, the lengths of these
groups can be take to exceed any lower bound given in advance.

\begin{proof}
Given $n$ and $\zeta ,$ fix $\xi >0$ and choose $K>\frac{n}{\xi }$ so that
for a set $X_{1}\subset X$ with $\mu \left( X_{1}\right) >\left( 1-\xi
\right) ,$ and for all $x\in X_{1},$ 
\begin{equation*}
\left\Vert dist_{T^{\left[ K\right] }x}\left( \bigvee_{i\in \left[ n\right]
}T^{-j}f\right) ,dist_{X}\left( \bigvee_{i\in \left[ n\right]
}T^{-j}f\right) \right\Vert _{\mathcal{M}}<\xi
\end{equation*}%
Choose finitely many disjoint sets $\left\{ A_{i}\right\} $, with $%
X_{2}:=\tbigcup\limits_{i}A_{i}\subset X_{1},$ and with $\mu \left(
X_{2}\right) >\left( 1-\xi \right) ,$ and so that for each $A_{i}$ and all $%
x,y\in A_{i},$ 
\begin{equation*}
\max_{0\leq j\leq K-1}\left\{ \rho \left( T^{j}x,T^{j}y\right) \right\} <\xi
.
\end{equation*}%
(The elements of $X_{2}$ are \textquotedblleft good $K-$points%
\textquotedblright\ whose $K-$orbits are \textquotedblleft good $K-$%
blocks\textquotedblright .) Choose $L$ so that most points have an $L-$orbit
which is mostly covered by good $K-$blocks, and which can therefore be $%
\left( 1-\xi \right) -$covered by disjoint good $K-$blocks. Moreover, we may
arrange that if these disjoint good $K-$blocks are partitioned into
\textquotedblleft types\textquotedblright\ according to the $A_{i}$ that
contains their initial element, then the each type repeats at least $\left( 
\frac{1}{\xi }\right) -$ many times in the given $L$ orbit.\newline
For each such repeated type of good $K-$block, cyclically divide the
occurrences of that type of block into consecutive $n-$blocks. That is,
divide the $j^{th}$ occurrence of the type into disjoint $n-$blocks starting
at position $\left[ j\right] _{n},$ where $\left[ i\right] _{n}\in \left\{
0,1,2,...,n\right\} $ and $\left[ j\right] _{n}\equiv j$ $\left( \text{mod }%
n\right) .$ If $\xi $ was chosen sufficiently small, the resulting
collection of $n-$ blocks covers $\left( 1-\zeta \right) $ of the $L$ orbit
by disjoint $n-$blocks with good $n-$distribution, and these $n-$blocks are
organized into consecutive groups nearly $K$ in length. Since the $K$ blocks
were good, if these groups of consecutive $n-$blocks are concatenated, the
resulting long block has $\zeta -$good $n-$dist.
\end{proof}
\end{lemma}

We note that lemma $\ref{model name 1}$ can immediately be strengthened so
that each of the points, whose existence is asserted by the lemma, has an $%
L- $name with $\zeta -$good $n-$distribution.

\begin{lemma}
\label{model name 2}(constructing a model name) Let $\left( T,X\right) $ be
an ergodic transformation and $f:X\rightarrow \left( M,\rho \right) $ a
measurable function$.$ For all $n\in \mathbb{N}$ and $\zeta >0$, and for all
sufficiently large $n_{1},$ and for arbitrarily large $L^{\prime },$ there
is a sequence $F\in M^{L^{\prime }}$ such that:

\begin{enumerate}
\item The $n_{1}-$distribution of $F$ is within $\zeta $ of the distribution
of $\bigvee_{i\in \left[ n_{1}\right] }T^{-i}f.$ That is,%
\begin{equation*}
\left\Vert dist_{\left[ L^{\prime }-n_{1}+1\right] }\left( F_{n_{1}}\right)
,dist_{X}\left( \bigvee_{i\in \left[ n_{1}\right] }T^{-i}f\right)
\right\Vert _{\mathcal{M}}<\zeta
\end{equation*}

\item $F$ is a union of consecutive $n_{1}-$blocks, and this set of $n_{1}-$%
blocks has $n_{1}-$distribution within $\zeta $ of the distribution of $%
\tbigvee\limits_{i=0}^{n_{1}-1}T^{-i}f.$ That is, if $I=\left\{ i\in \left[
L^{\prime }\right] \mid i\equiv 0\left( \text{mod }n_{1}\right) \right\} ,$
then%
\begin{equation*}
\left\Vert dist_{I}\left( F_{n_{1}}\right) ,dist_{X}\left( \bigvee_{i\in 
\left[ n_{1}\right] }T^{-i}f\right) \right\Vert _{\mathcal{M}}<\zeta
\end{equation*}

\item Each of the disjoint $n_{1}$ blocks above is at least $\left( 1-\zeta
\right) -$covered by a set of disjoint $n-$blocks which has $n-$distribution
within $\zeta $ of the distribution of $\bigvee_{i\in \left[ n\right]
}T^{-i}f.$
\end{enumerate}

\begin{proof}
Given $\left( n,\zeta \right) ,$ choose $\zeta _{1}>0$ and let $n_{1}\geq
L\left( n,\zeta _{1}\right) $ (as defined in lemma $\ref{model name 1}$).
Choose $L^{\prime }>L\left( n_{1},\zeta _{1}\right) $ so that, in addition,
most points have $L^{\prime }$ names with $\zeta _{1}-$good $n_{1}-$%
distribution. Fix such a point $x\in X.$ Cover (a $\left( 1-\zeta
_{1}\right) -$fraction of) its $L^{\prime }-$orbit by a set of disjoint $%
n_{1}-$blocks which has $\zeta _{1}-$good $n_{1}-$distribution, and which
blocks are organized into groups of consecutive blocks, where each of which
group (as a single sequence) has $\zeta _{1}-$good $n_{1}-$distribution. So
most of these (disjoint) $n_{1}-$blocks can be $\left( 1-\zeta _{1}\right) -$
covered, disjointly, by a set of $n-$blocks with $\zeta -$good $n-$%
distribution. Throw out the $\zeta _{1}-$fraction of $n_{1}-$blocks that
can't be so covered, and throw out the $\zeta _{1}-$fraction of the orbit
between the groups of consecutive $n_{1}-$blocks, and then push these
remaining $n_{1}-$blocks together. If $\zeta _{1}$ was chosen sufficiently
small, this (modified) orbit has the name we want.
\end{proof}
\end{lemma}

\begin{lemma}
\label{convex}(convexity lemma) Let $\left( V,\left\Vert \text{ }\right\Vert
\right) $ be a normed real vector space, and suppose that $v_{1},v_{2}$ and $%
v_{Q}\in V$ and $0<\zeta \leq \varepsilon .$ If $v_{Q}=\left( 1-\varepsilon
\right) v_{1}+\varepsilon v_{2}$ and $\left\Vert v_{1}-v_{Q}\right\Vert
<\zeta ,$ then $\left\Vert v_{2}-v_{Q}\right\Vert <\frac{\zeta }{\varepsilon 
}.$

\begin{proof}
We have $v_{Q}=v_{1}+\varepsilon \left( v_{2}-v_{1}\right) ,$ so $%
v_{Q}-v_{1}=\varepsilon \left( v_{2}-v_{1}\right) ,$ so $\left\Vert
v_{Q}-v_{1}\right\Vert =\varepsilon \left\Vert v_{2}-v_{1}\right\Vert .$
Similarly $\left\Vert v_{2}-v_{Q}\right\Vert =\left( 1-\varepsilon \right)
\left\Vert v_{2}-v_{1}\right\Vert .$ So%
\begin{equation*}
\left\Vert v_{Q}-v_{2}\right\Vert =\frac{\left( 1-\varepsilon \right) }{%
\varepsilon }\left\Vert v_{Q}-v_{1}\right\Vert <\frac{\left( 1-\varepsilon
\right) }{\varepsilon }\zeta <\frac{\zeta }{\varepsilon }.
\end{equation*}
\end{proof}
\end{lemma}

This lemma will be applied to probabity vectors $v_{1},v_{2}$ and $v_{Q}$
viewed as elements of $\mathbb{R}^{t}$ with respect to the $l^{1}-$norm,
where $v_{Q}$ will be the distribution of a $t-$set partition $Q$ on a
probability space, and $v_{1}$ and $v_{2}$ will be the conditional
distributions of $Q$ on subsets of measure $1-\varepsilon $ and $\varepsilon
,$ respectively.

\begin{lemma}
\label{sampling}(Sampling lemma) For all $n\in \mathbb{N}$, $\delta \in
\left( 0,2^{-n}\right) $ and $\zeta >0$ there exists $K=K\left( n,\delta
,\zeta \right) \in \mathbb{N}$ so that given any set $E$ with $\left\vert
E\right\vert \leq 2^{n}$ and any probability measure $\nu $ on $E$ such that
for all $e\in E,$ $\nu \left( e\right) >\delta ,$ and any set $D$ with $%
\left\vert D\right\vert \geq K,$ there exists a function $f$ from $D$ \emph{%
onto} $E$ so that $\left\Vert dist_{D}\left( f\right) -\nu \right\Vert _{%
\mathcal{M}}<\zeta .$

\begin{proof}
Fix $n\in \mathbb{N}$, $\delta >0$ and $\zeta >0.$ Choose $K\in \mathbb{N}$
with $\frac{1}{K}<\min \left\{ \delta ,\frac{\zeta }{2^{n}}\right\} .$
Suppose we are given a set $E$ and measure $\nu $ as above, and a set $D$
with $\left\vert D\right\vert =K^{\prime }\geq K.$ Partition $\left[ 0,1%
\right] $ into subintervals whose lengths equal the measures of the atoms of 
$\nu .$ That is, partition $\left[ 0,1\right] $ by $\left\{
0=x_{0}<x_{1}<...<x_{t}=1\right\} $ so that $\left\vert
x_{i}-x_{i-1}\right\vert =\nu \left( e_{i}\right) ,$ where $e_{i}$ is the $%
i^{th}$ element of $E.$ Modify this partition by moving each endpoint $x_{i}$
to the nearest multiple of $\frac{1}{K^{\prime }}$ below it. This new
partition determines a distribution $\nu _{1}$ that is $\zeta -$close to $%
\nu .$ But there is a function $f:D\rightarrow E$ whose statistical
distribution is exactly $\nu _{1}$.
\end{proof}
\end{lemma}

\begin{lemma}
\label{exhaustion}(Exhaustion lemma) Suppose that $\delta ^{\prime }>0,$and $%
\varepsilon >0$ are given, and suppose $\zeta <\frac{\varepsilon \delta
^{\prime }}{2}.$ Then for all $K^{\prime }$ there exists $N^{\prime }\in 
\mathbb{N}$ such that if $\left( Z,\lambda \right) $ is a discrete
probability space with normalized counting measure $\lambda $ such that $%
\left\vert Z\right\vert >N^{\prime }$ and $Q$ is a finite partition of $Z$,
each of whose atoms has $\lambda -$measure at least $\delta ^{\prime }$, and
if $\left\{ S_{i}\subset Z\right\} _{i=1}^{r}$ is a pairwise disjoint
(non-empty) sequence of subsets of $Z$ such that for all $i$ and $j,$ $%
\left\vert S_{i}\right\vert =K^{\prime },dist_{S_{i}}Q=dist_{Sj}Q,$ $%
\left\Vert dist_{S_{i}}Q-dist_{Z}Q\right\Vert _{\mathcal{M}}<\zeta $ and $%
\lambda \left( \bigcup_{i=1}^{r}S_{i}\right) <1-\varepsilon ,$ then there is
an additional set $S_{r+1}\subset Z\backslash \bigcup_{i=1}^{r}S_{i},$ with $%
\left\vert S_{i}\right\vert =K^{\prime },$ on which $%
dist_{S_{r+1}}Q=dist_{S_{1}}Q.$

\begin{proof}
Choose $N^{\prime }>\frac{K^{\prime }}{\left( \frac{\varepsilon \delta
^{\prime }}{2}\right) }.$ Suppose that $\left( Z,\lambda \right) $ and $Q$
and $\left\{ S_{i}\right\} _{i=1}^{r}$ are as in the statement of the lemma.
Writing $S=\bigcup_{i=1}^{r}S_{i},$ we would have $\lambda \left(
Z\backslash S\right) =\varepsilon ^{\prime }>\varepsilon ,$ and by lemma $%
\ref{convex},$ 
\begin{equation*}
\left\Vert dist_{Z\backslash S}Q,dist_{Z}Q\right\Vert _{\mathcal{M}}<\frac{%
\zeta }{\varepsilon ^{\prime }}<\frac{\zeta }{\varepsilon }.
\end{equation*}%
Since $\frac{\zeta }{\varepsilon }<\frac{\delta ^{\prime }}{2}$, each atom
of the trace of $Q$ on $Z\backslash S$ has conditional measure at least $%
\frac{\delta ^{\prime }}{2}$ and (unconditional) measure at least $\left( 
\frac{\varepsilon ^{\prime }\delta ^{\prime }}{2}\right) .$ Since $\left( 
\frac{\varepsilon ^{\prime }\delta ^{\prime }}{2}\right) N^{\prime }>\left( 
\frac{\varepsilon \delta ^{\prime }}{2}\right) N^{\prime }>K^{\prime },$
each atom of the trace of $Q$ on $Z\backslash S$ has at least $K^{\prime }$
elements. Therefore there is an injection $g:S_{1}\rightarrow Z\backslash S$
so that for all $s\in S_{1},$ $Q\left( g\left( s_{1}\right) \right) =Q\left(
s_{1}\right) ,$ and setting $g\left( S_{1}\right) =S_{r+1}$ completes the
proof.
\end{proof}
\end{lemma}

\begin{remark}
The above lemma says that if we are planning to take samples from a discrete
uniform measure space which is partitioned by $Q,$ where $Q$ has finitely
many atoms and none of very small measure, and if we are planning to do so
using samples of a known size ($K^{\prime }$), then if the samples will have
distribution close enough (within $\zeta $) to that of $Q,$ and if the
discrete space is large enough compared to $K^{\prime },$ we will be able to
take repeated samples (without replacement) until the space is nearly
exhausted (to within preassigned $\varepsilon ).$
\end{remark}

\subsection{Weak topology}

We collect here some basic facts about the metric $\left\Vert _{\qquad
}\right\Vert _{\mathcal{M}}$ and the topology it generates. As before, $G$
denotes a compact group with Haar measure $\lambda $ and two-sided invariant
metric $\rho .$

\begin{lemma}
\label{trans unif}Let $\varepsilon >0$ and $n\in \mathbb{N}$, and suppose
that the sequence $\gamma :\left[ n\right] \rightarrow G$ satisfies%
\begin{equation*}
\left\Vert dist_{\left[ n\right] }\gamma ,\lambda \right\Vert _{\mathcal{M}%
}<\varepsilon .
\end{equation*}%
Then for all $h\in g,$%
\begin{equation*}
\left\Vert dist_{\left[ n\right] }\gamma h,\lambda \right\Vert _{\mathcal{M}%
}<\varepsilon .
\end{equation*}

\begin{proof}
The proof is immediate.
\end{proof}
\end{lemma}

\begin{lemma}
\label{wk conv}Let $A$ be an open subset of $G.$ Then for all $\varepsilon
>0 $ there exists $\eta >0$ such that for all $n\in \mathbb{N}$, if $\gamma :%
\left[ n\right] \rightarrow G$ is a (finite) sequence in $G$ such that 
\begin{equation*}
\left\Vert dist_{\left[ n\right] }\gamma ,\lambda \right\Vert _{\mathcal{M}%
}<\eta
\end{equation*}%
then%
\begin{equation*}
\frac{1}{n}\sum_{i\in \left[ n\right] }\chi _{A}\left( \gamma \left(
i\right) \right) \geq \lambda \left( A\right) -\varepsilon
\end{equation*}

\begin{proof}
Using the fact that $\left\Vert _{\qquad }\right\Vert _{\mathcal{M}}$
metrizes the weak topology \cite{D1}, the conclusion is a statement of a
well-known fact about the weak topology, (see \cite{D2}) and holds in
general for arbitrary probability measures on a separable metric space $%
\left( M,\rho \right) $ of finite diameter in the place of $dist_{\left[ n%
\right] }\gamma \ $and $\lambda $ on $G.$
\end{proof}
\end{lemma}

Combining the previous two lemmas gives us the following

\begin{lemma}
\label{wk conv 2}Let $A$ be an open subset of $G.$ Then for all $\varepsilon
>0$ there exists $\eta >0$ such that for all $n\in \mathbb{N}$, if $\gamma :%
\left[ n\right] \rightarrow G$ is a sequence such that 
\begin{equation*}
\left\Vert dist_{\left[ n\right] }\gamma ,\lambda \right\Vert _{\mathcal{M}%
}<\eta
\end{equation*}%
then for all $h\in G$%
\begin{equation*}
\frac{1}{n}\sum_{i\in \left[ n\right] }\chi _{A}\left( \gamma \left(
i\right) h\right) \geq \lambda \left( A\right) -\varepsilon
\end{equation*}
\end{lemma}

\begin{definition}
Let $\nu $ be a Borel probability measure on a metric space $\left( M,\rho
\right) .$ A set $A\subset M$ is called a \emph{continuity set for }$\nu $
if $\nu \left( \bar{A}\backslash A^{o}\right) =0.$
\end{definition}

We note that for all $x\in M$ and all $\delta >0$ there exists $\delta
^{\prime }<\delta $ such that the ball $B_{\delta ^{\prime }}\left( x\right) 
$ is a continuity set. This is because at most countably many of the
pairwise disjoint circles $\left\{ x^{\prime }\mid \rho \left( x,x^{\prime
}\right) =\delta ^{\prime }\right\} $ can have positive measure. Therefore,
if $M$ is compact, then for every $\delta >0$ there is a finite partition of 
$M$ into continuity sets of diameter less than $\delta .$

The following lemma follows quickly from lemma \ref{wk conv} (which applies
to more general metric spaces, as we've indicated).

\begin{lemma}
\label{continuity sets}Let $\nu $ be a Borel probability measure on a
compact metric space $\left( M,\rho \right) .$ Let $Q=\left\{
Q_{1},...,Q_{t}\right\} $ be a finite partition of $M$ into continuity sets
for $\nu .$ Then for all $\zeta >0$ there exists $\tilde{\zeta}>0$ so that
if $\gamma :\left[ n\right] \rightarrow M$ is a sequence with 
\begin{equation*}
\left\Vert dist_{\left[ n\right] }\gamma ,\nu \right\Vert _{\mathcal{M}}<%
\tilde{\zeta}
\end{equation*}%
then 
\begin{equation*}
\left\Vert dist_{\left[ n\right] }\left( Q\left( \gamma \right) \right)
,distQ\right\Vert _{\mathcal{M}}<\zeta
\end{equation*}
\end{lemma}

We will occasionally need to implement a distribution match in a concrete
way. The following lemmas allow us to do this.

\begin{lemma}
\label{dist match 1}Let $\nu $ be a Borel probability measure on a compact
metric space $\left( M,\rho \right) .$ Then for all $\zeta \in \left(
0,1\right) $ there exists $\tilde{\zeta}>0$ so that if $\gamma _{1},\gamma
_{2}:\left[ n\right] \rightarrow M$ are sequences such that, for both $i=1$
and $2,$ 
\begin{equation*}
\left\Vert dist_{\left[ n\right] }\gamma _{i},\nu \right\Vert _{\mathcal{M}}<%
\tilde{\zeta}
\end{equation*}%
then there is a bijection $\phi :\left[ n\right] \rightarrow \left[ n\right] 
$ such that for $\left( 1-\zeta \right) -$most $i\in \left\{
1,2,...,n\right\} ,$ we have 
\begin{equation*}
\rho \left( \gamma _{1}\left( \phi \left( i\right) \right) ,\gamma
_{2}\left( i\right) \right) <\zeta .
\end{equation*}

\begin{proof}
Fix a finite partition $Q=\left\{ Q_{1},...,Q_{t}\right\} $ of $M$ into
continuity sets for $\nu .$ Choose $\tilde{\zeta}$ by lemma \ref{continuity
sets} with respect to $\nu ,Q$ and $\frac{\left( \zeta \right) ^{2}}{2}.$
Suppose $\gamma _{1}$ and $\gamma _{2}$ meet the conditions of this lemma.
Then%
\begin{equation*}
\left\Vert dist_{\left[ n\right] }\left( Q\left( \gamma _{1}\right) \right)
,dist_{\left[ n\right] }\left( Q\left( \gamma _{2}\right) \right)
\right\Vert _{\mathcal{M}}<\left( \zeta \right) ^{2}
\end{equation*}%
The conclusion follows from this.
\end{proof}
\end{lemma}

\begin{lemma}
\label{dist match 2}Let $\nu $ be a Borel probability measure on a compact
metric space $\left( M,\rho \right) .$ Then for all $\zeta \in \left(
0,1\right) $ and $n\in \mathbb{N},$ there exists $\tilde{\zeta}$ and $N$ so
that for all $n_{1}>N,$ and for all sequences $\gamma _{1}:\left[ n\right]
\rightarrow M$ and $\gamma _{2}:\left[ n_{1}\right] \rightarrow M,$ such
that, for both $i=1$ and $2,$ 
\begin{equation*}
\left\Vert dist\gamma _{i},\nu \right\Vert _{\mathcal{M}}<\tilde{\zeta}
\end{equation*}%
there is a map $\phi :\left[ n_{1}\right] \rightarrow \left[ n\right] $ so
that for $\left( 1-\zeta \right) -$most $i\in \left[ n\right] ,$ we have 
\begin{equation*}
\rho \left( \gamma _{1}\left( \phi \left( i\right) \right) ,\gamma
_{2}\left( i\right) \right) <\zeta
\end{equation*}%
and for each $i\in \left[ n\right] ,$%
\begin{equation*}
\left\vert \frac{\left\vert \phi ^{-1}\left( i\right) \right\vert }{n_{1}}-%
\frac{1}{n}\right\vert <\zeta .
\end{equation*}%
(The last condition can be interpreted as saying $\phi $ nearly preserves
the normalized counting measures on $\left[ n_{1}\right] $ and $\left[ n%
\right] $).

\begin{proof}
The proof is similar to the proof of lemma \ref{dist match 1}.
\end{proof}
\end{lemma}

We will refer to the maps $\phi $ of lemmas \ref{dist match 1} and \ref{dist
match 2} as $\zeta -$distribution matches between the sequences $\gamma _{1}$
and $\gamma _{2}.$

We will also need the following simple observation.

\begin{lemma}
\label{stable unif}Suppose that $\gamma :\left[ n\right] \rightarrow G$ is a
(finite) sequence in $G$ such that 
\begin{equation*}
\left\Vert dist_{\left[ n\right] }\gamma ,\lambda \right\Vert _{\mathcal{M}%
}<\eta
\end{equation*}%
and $\alpha :\left[ n\right] \rightarrow G$ is a sequence such that for all $%
i,$ $\rho \left( \alpha \left( i\right) ,id_{G}\right) <\eta .$ Then%
\begin{equation*}
\left\Vert dist_{\left[ n\right] }\alpha \gamma ,\lambda \right\Vert _{%
\mathcal{M}}<2\eta
\end{equation*}
\end{lemma}

\subsection{Rokhlin lemma and ergodicity}

Our argument will depend in an essential way on the Rokhlin lemma. In
particular, we will make use of Rokhlin towers in $G-$extensions, where the
towers are measurable with respect to the base factor.

\begin{definition}
Given a $G-$extension $\left( S,T,X,\sigma \right) ,$ a \emph{Rokhlin tower
measurable with respect to the base factor }$\left( T,\mathcal{A}\right) $
is a pairwise disjoint sequence of sets $R=\left\{ S^{i}B\right\} _{i\in %
\left[ K\right] },$ where each $S^{i}B\in \mathcal{A}$. The set $B$ is
called the $\emph{base}$ of the tower and $K$ its \emph{height}. If $\mu
\times \lambda \left( \bigcup_{i\in \left[ K\right] }S^{i}B\right) >1-\zeta
, $ we call $R$ a $\left( 1-\zeta \right) -K-$\emph{tower}. If $P$ is an $%
\mathcal{A}-$measurable partition, then a $P-$\emph{column of }$R$ is a
sequence $C=\left\{ S^{i}B^{\prime }\right\} _{i\in \left[ K\right] },$
where $B^{\prime }\subset B$ is an atom of the trace of $\bigvee_{i\in \left[
K\right] }S^{-i}P$ on $B.$ The sets $S^{i}B^{\prime }$ are referred to as 
\emph{levels} of the column $C$. A sequence of sets of the form $\left\{
S^{i}L\right\} _{i\in \left[ k\right] }$ where $L$ is a level of a column $C$
is called a \emph{column-block of }$C$ (of length $k$)$.$
\end{definition}

The term \textit{column-block} is used to emphasize the distinction between
a block consisting of levels of a column and a block consisting of points in
an orbit, when both are in play together during our construction below. To
be specific, in the arguments to follow we will have occasion to construct
Rokhlin towers of the above type with respect to a $G-$extension $\bar{S}%
_{0}=\left( \bar{T}_{0},\bar{X},\bar{\sigma}\right) ,$ but in the presence
of a $G-$speedup $\bar{S}=\bar{S}_{0}^{k}$ of $\bar{S}_{0},$ where $k:\bar{X}%
\rightarrow \mathbb{N}.$ Columns in these towers will be constructed so that 
$k$ is constant on every level, so that we can speak of column blocks that
are consecutive images of a level under the speedup $\bar{S}$ as opposed to $%
\bar{S}_{0}.$ In this case we will speak of $\bar{S}-$\emph{column-blocks},
to distinguish them from $\bar{S}_{0}-$\emph{column-blocks}$.$

All the language introduced above concerning Rokhlin towers, blocks and
column-blocks will apply in an obvious way to partial transformations. As
before, a prefix may be attached whenever we need to distinguish objects
associated with a tranformation $S_{0}$ from those associated with a speedup 
$S$ of $S_{0}.$

The Rokhlin lemma can be formulated as follows.

\begin{lemma}
\label{strR}Let $T$ be an ergodic measure preserving transformation of $%
\left( X,\mu \right) $, and $f$ a measurable function from $X$ to the metric
space $\left( M,\rho \right) $. Then for all $K\in \mathbb{N}$ and $%
\varepsilon >0$ there is a $\left( 1-\varepsilon \right) -K-$tower $%
R=\left\{ S^{i}B\right\} _{i=0}^{K-1}$ such that 
\begin{equation*}
\left\Vert dist_{B}\left( f\right) ,dist_{X}\left( f\right) \right\Vert _{%
\mathcal{M}}<\varepsilon .
\end{equation*}
\end{lemma}

We will need to arrange that the speedups we construct are ergodic. To do
this we will use the following criterion for ergodicity. Recall that a
transformation $T^{\prime }$ is said to be in the full group of $T$ if each
orbit of $T^{\prime }$ is contained in an orbit of $T.$

\begin{lemma}
\label{ergodicity}Fix a sequence $\left\{ C_{i}\right\} _{i=1}^{\infty }$ of
measurable sets in the probability space $\left( X,\mathcal{A},\mu \right) $
such that the algebra they generate is dense in the measure algebra of $%
\left( X,\mathcal{A},\mu \right) $. Suppose that $T$ is a transformation of $%
\left( X,\mathcal{A},\mu \right) $ and for all $i$ and $j$ such that $\mu
\left( C_{i}\right) <\mu \left( C_{j}\right) $ and for all $\varepsilon >0$
there is a transformation $T^{\prime }$ in the full group of $T$ such that $%
\mu \left( C_{j}\cap T^{\prime }\left( C_{i}\right) \right) >\left(
1-\varepsilon \right) \mu \left( C_{i}\right) .$ Then $T$ is ergodic.
\end{lemma}

\section{Basic iterative procedure}

The key argument of the proof of our theorems is contained in the following
Distribution Improvement Lemma. This lemma shows that, given a partial $G-$%
speedup $\bar{S}$ of a $G-$extension $\bar{S}_{0}$, which approximates a $G-$%
extension $S,$ we can make a small modification of $\bar{S}$ to obtain a
partial $G-$speedup that is a much improved approximation of $S.$ The rest
of this section will be devoted to proving this lemma. The reader familiar
with Ornstein's proof of the isomorphism theorem for Bernoulli shifts will
recognize this as the counterpart of the \textquotedblleft fundamental
lemma\textquotedblright\ of that argument. The repeated application of this
lemma will quickly lead to proofs of the theorems we want, and these will be
found in the final section of the paper.

Before formulating the basic lemma, it will be convenient to introduce some
new language. First we describe the basic scheme by which orbits of $\bar{S}%
_{0}$ will be manipulated to obtain a partial speedup. This is a purely
combinatorial construction that we describe in terms of sequences of
integers.

Recall that for $n\in \mathbb{N}$, $\left[ n\right] $ denotes $\left\{
0,1,...,n-1\right\} .$ More generally, for $r\in \mathbb{R}^{\geq 0}$ we let 
$\left[ r\right] =\mathbb{Z\cap \lbrack }0,r).$

Let $M<M^{\prime }$ and $w\ $be elements of $\mathbb{N}.$ Suppose that $u:%
\left[ w\right] \rightarrow \left[ M^{\prime }-M+1\right] \ $such that for
all $s\in \left[ w\right] ,$ $u\left( s+1\right) -u\left( s\right) \geq M.$
For all $s\in \left[ w\right] $ define $\tilde{W}_{s}:\left[ M\right]
\rightarrow \left[ M^{\prime }\right] $ by $\tilde{W}_{s}\left( j\right)
=u\left( s\right) +j.$ We refer to $\left\{ \tilde{W}_{s}\right\} _{s\in %
\left[ w\right] }$ as a system of $w$ \emph{windows} of length $M$ in $\left[
M^{\prime }\right] .$

Suppose that $p\in \left[ w\right] $, and suppose that for each $l\in \left[
p\right] $ and each $j\in \left[ \frac{w-l}{p}\right] $ we are given $%
t_{j}^{l}:\left[ p\right] \rightarrow \left[ M\right] .$ Then we obtain $%
g_{j}^{l}:\left[ p\right] \rightarrow \left[ M^{\prime }\right] $ given by 
\begin{equation*}
g_{j}^{l}\left( i\right) =\tilde{W}_{jp+l+i}\left( t_{j}^{l}\left( i\right)
\right) .
\end{equation*}%
Suppose further that map $\left( l,j,i\right) \mapsto g_{j}^{l}\left(
i\right) $ is injective.

\begin{definition}
We refer to such a family $\Gamma =\left\{ g_{j}^{l}\right\} $ of increasing
subsequences of $\left[ M^{\prime }\right] $ as a \emph{cycle }$\left( \text{%
\emph{of}}\emph{\ }p-\text{sequences in }\left[ M^{\prime }\right] \right) $%
. For each $l,$ we refer to $\left\{ g_{j}^{l}\right\} _{j\in \left[ \frac{%
w-l}{p}\right] }$ as the $l^{th}$ \emph{pass} (through $\left[ M^{\prime }%
\right] $) of $\Gamma $. For each $l$ and $j$ we refer to the sequence $%
g_{j}^{l}$ as the $j^{th}$ \emph{stage} of the $l^{th}$ pass of $\Gamma $.
\end{definition}

We describe this informally: in each stage of the cycle $\Gamma $ the
function $g_{j}^{l}$ selects one point from each of a sequence of $p$
successive windows\ in $\left[ M^{\prime }\right] .$ The $l^{th}$ pass of $%
\Gamma $ is a sequence of stages, whose $0^{th}$ stage begins in window $%
\tilde{W}_{l},$ whose every stage begins at the window immediately after the
last window of the previous stage, and where as many stages are completed as
the sequence of windows can accommodate. When this scheme is implemented
below, the sequence $\left[ M^{\prime }\right] $ will correspond to an orbit
block of a transformation, and each $g_{j}^{l}$ will identify an orbit of a
partial speedup of that transformation.

We note that, for all $s\in \left[ w\right] ,$ 
\begin{equation*}
\left\vert \left\{ \left( l,j,i\right) \mid l\in \left[ p\right] ,j\in \left[
\frac{w-l}{p}\right] ,i\in \left[ p\right] ,\text{ and }jp+l+i=s\right\}
\right\vert \leq p
\end{equation*}%
and more importantly, if $\left\vert s\right\vert \geq p-1$ and $\left\vert
w-s\right\vert \geq p-1,$ then the above cardinality equals $p$, and 
\begin{equation*}
\left\{ i\in \left[ p\right] \mid \left( \exists l,j\right) \text{ such that 
}jp+l+i=s\right\} =\left[ p\right]
\end{equation*}%
Indeed, if $\left\vert s\right\vert \geq p-1$ and $\left\vert w-s\right\vert
\geq p-1,$ then on pass $l,$ $g_{j}^{l}\left( i\right) $ lies in the range
of $W_{s}$ for exactly one value $i,$ and in the subsequent pass $l+1,$ the
corresponding value of $i$ is one less $\left( \func{mod}p\right) $.

Next we describe the special form that each of the partial speedups that we
construct will have. To describe this form, suppose that $\left( \bar{S},%
\bar{T},\bar{\sigma},\bar{X}\right) $ is a partial $G-$speedup of a $G-$%
extension $\left( \bar{S}_{0},\bar{T}_{0},\bar{\sigma}_{0},\bar{X}\right) $
on a space $\left( \bar{X}\times G\right) .$ Let $k:Dom\left( \bar{T}\right)
\rightarrow \mathbb{N}$ denote the measurable function such that for every $%
\bar{x}\in Dom\left( \bar{T}\right) ,$ $\bar{S}\left( \bar{x},g\right) =\bar{%
S}_{0}^{k\left( \bar{x}\right) }\left( \bar{x},g\right) $ and $\bar{\sigma}=%
\bar{\sigma}_{0}^{\left( k\left( \bar{x}\right) \right) }$ the associated
\textquotedblleft skewing\textquotedblright\ function as in $\left( \ref%
{sigma cocycle}\right) .$

\begin{definition}
If $\bar{P}$ is a measurable partition of $\bar{X}$, we will say that the
pair $\left( \bar{S},\bar{P}\right) $ is a \emph{regular}\ partial $G-$%
speedup of $\bar{S}_{0}$ if the following conditions are met:

\begin{enumerate}
\item There is a set $\bar{B},$ measurable with respect to $\bar{X},$ such
that for some $L\in \mathbb{N},$ the sets $\left\{ \bar{S}^{i}\bar{B}%
\right\} _{i\in \left[ L\right] }$ are disjoint, and the domain of $\bar{S}$
is precisely $\bigcup_{i\in \left[ L-1\right] }\bar{S}^{i}\bar{B}.$ (We
refer to $\bar{R}=\tbigcup\limits_{i\in \left[ L\right] }\bar{S}^{i}\bar{B}$
as the \textit{speedup tower} for $\bar{S}$ and to $\bar{B}$ as its\textit{\
base}).

\item $k$ is bounded

\item For all $\left( \bar{x},g\right) ,\left( \bar{x}^{\prime },g\right)
\in \bar{B},$%
\begin{equation*}
\bigvee_{i\in \left[ L\right] }\bar{S}^{-i}\left( \bar{P}\vee \bar{c}\right)
\left( \bar{x},g\right) =\bigvee_{i\in \left[ L\right] }\bar{S}^{-i}\left( 
\bar{P}\vee \bar{c}\right) \left( \bar{x}^{\prime },g\right) .
\end{equation*}%
(Recall that $\bar{c}$ denotes the projection on the $G-$coordinate).
\bigskip \newline
If, in addition, for some $n\in \mathbb{N}$ and $\delta >0$ we have the
further properties that\bigskip

\item $L$ is a multiple of $n,$ and for each $\left( \bar{x},g\right) \in 
\bar{B},$ 
\begin{equation*}
\left\Vert dist_{\bar{S}^{\left[ \frac{L}{n}\right] n}\left( \bar{x}%
,g\right) }\bigvee_{i\in \left[ n\right] }\bar{S}^{-i}\left( \bar{P}\vee 
\bar{c}\right) ,dist_{Dom\left( \bar{S}^{n}\right) }\bigvee_{i\in \left[ n%
\right] }\bar{S}^{-i}\left( \bar{P}\vee \bar{c}\right) \right\Vert _{%
\mathcal{M}}<\delta .
\end{equation*}%
That is, when the $\bar{S}-L-$orbit of $\left( \bar{x},g\right) $ is divided
into disjoint, consecutive $n-$blocks, those $n-$blocks have a distribution
of names that is $\delta -$close to the full $n-$distribution of the speedup 
$\bar{S},$ and

\item $\bar{\mu}\times \lambda \left( Dom\left( \bar{S}\right) \right)
>1-\delta .$\newline
then we will say that the pair $\left( \bar{S},\bar{P}\right) $ is $\left(
n,\delta \right) -$\emph{regular}.
\end{enumerate}
\end{definition}

We note that if $\mu \times \lambda \left( Dom\left( \bar{S}\right) \right)
>1-\delta ,$ then we must have $L>\frac{1-\delta }{\delta }.$ (If $l$ is the
measure of a single level of the speedup tower, then $l<\delta ,$ but $%
\left( L-1\right) l>1-\delta $ so $L>L-1>\frac{1-\delta }{l}>\frac{1-\delta 
}{\delta }.$)

\begin{definition}
If $\left( \bar{S},\bar{P}\right) $ is an $\left( n,\delta \right) -$regular
partial $G-$speedup of $\bar{S}_{0}\ $as above, we refer to the set 
\begin{equation*}
\Lambda _{n}\left( \bar{S}\right) =\bigcup_{i\in \left[ \frac{L}{n}\right] }%
\bar{S}^{in}\left( \bar{B}\right)
\end{equation*}%
as the $n-$\emph{ladder} of $\bar{S}.$ We refer to a block of the form $\bar{%
S}^{\left[ n\right] }\left( \bar{x},g\right) ,$ where $\left( \bar{x}%
,g\right) \in \Lambda _{n}\left( \bar{S}\right) $ as a \emph{ladder block}
of $\bar{S}$. Suppose that $\left( \bar{x}^{\prime },g^{\prime }\right) $ is
a point in the ladder block $\bar{S}^{\left[ n\right] }\left( \bar{x}%
,g\right) ,$ and $\tilde{S}$ is another partial transformation on $\bar{X}%
\times G.$ We say that the ladder block of $\left( \bar{x}^{\prime
},g^{\prime }\right) $ is \emph{broken by }$\tilde{S}$ if for some $i\in %
\left[ 0,n-2\right] ,$ $\bar{S}\left( \bar{S}^{i}\left( \bar{x},g\right)
\right) \neq \tilde{S}\left( \bar{S}^{i}\left( \bar{x},g\right) \right) .$
\end{definition}

\begin{lemma}
\label{dil}(Distribution Improvement Lemma) For all $\varepsilon >0$ and for
all open $A_{2}\subset G,$ there exist $\delta >0$ and $n\in \mathbb{N}$
such that, if $\left( S,T,\sigma ,X\right) $ and $\left( \bar{S}_{0},\bar{T}%
_{0},\bar{\sigma}_{0},\bar{X}\right) $ are\ ergodic $G-$extensions on $%
\left( X\times G\right) $ and $\left( \bar{X}\times G\right) $ respectively,
and $\left( \bar{S},\bar{T},\bar{\sigma},\bar{X}\right) $ is a partial $G-$%
speedup of $\bar{S}_{0}$ and $P$ and $\bar{P}$ are finite partitions of $X$
and $\bar{X}$ such that $\left( \bar{S},\bar{P}\right) $ is $\left( n,\delta
\right) -$regular, such that 
\begin{equation}
\left\Vert dist_{X\times G}\bigvee_{i\in \left[ n\right] }S^{-i}\left( P\vee
c\right) ,dist_{Dom\left( \bar{S}^{n}\right) }\bigvee_{i\in \left[ n\right] }%
\bar{S}^{-i}\left( \bar{P}\vee \bar{c}\right) \right\Vert _{\mathcal{M}%
}<\delta  \label{delta dist match}
\end{equation}%
then for all $\delta _{1}>0$ and all sufficiently large $n_{1},$ and all $%
A_{1}\in \mathcal{\bar{A}}$, there is a partial\ $G-$speedup $\bar{S}_{1}$
of $\bar{S}_{0}$ and a partition $\bar{P}_{1}$ and a measurable function $%
\bar{\alpha}:\bar{X}\rightarrow G$ such that 
\begin{equation}
\left( \bar{S}_{1}^{\bar{\alpha}},\bar{P}_{1}\right) \text{ is }\left(
n_{1},\delta _{1}\right) -\text{regular,}  \label{n_1,d_1 regular}
\end{equation}%
\begin{equation}
\left\vert \bar{P}-\bar{P}_{1}\right\vert <\varepsilon ,  \label{close P}
\end{equation}%
\begin{equation}
\int \rho \left( \bar{\alpha}\left( \bar{x}\right) ,id_{G}\right) d\bar{\mu}%
\left( \bar{x}\right) <\varepsilon ,  \label{small alpha}
\end{equation}%
if $D$ denotes the set of points in the speedup tower of $\bar{S}$ such
whose ladder block is broken by $\bar{S}_{1},$ then%
\begin{equation}
\bar{\mu}\times \lambda \left( D\right) <\delta _{1},  \label{small break}
\end{equation}%
\begin{equation}
\left\Vert dist_{X\times G}\bigvee_{i\in \left[ n_{1}\right] }S^{-i}\left(
P\vee c\right) ,dist_{Dom\left( \bar{S}_{1}^{n_{1}}\right) }\bigvee_{i\in %
\left[ n_{1}\right] }\bar{S}_{1}^{-i}\left( \bar{P}_{1}\vee \bar{\alpha}\bar{%
c}\right) \right\Vert _{\mathcal{M}}<\delta _{1},  \label{dist match}
\end{equation}%
and setting $A=A_{1}\times A_{2},$ the set of $y\in \Lambda _{n_{1}}\left( 
\bar{S}_{1}\right) $ such that 
\begin{equation}
\frac{1}{n_{1}}\sum_{i\in \left[ n_{1}\right] }\boldsymbol{1}_{A}\left( \bar{%
S}_{1}^{i}\left( y\right) \right) >\left( \bar{\mu}\times \lambda \right)
\left( A\right) -\varepsilon  \label{good A}
\end{equation}%
has measure greater than $\left( 1-\varepsilon \right) \left( \bar{\mu}%
\times \lambda \right) \left( \Lambda _{n_{1}}\left( \bar{S}_{1}\right)
\right) .$
\end{lemma}

\begin{proof}
Fix $\varepsilon >0$ and an open set $A_{2}\subset G.$ Choose $\varepsilon
^{\prime }$ as in lemma \ref{wk conv 2} with respect to $\frac{\varepsilon }{%
100}$ and $A_{2}.$ (We may assume that $\varepsilon ^{\prime }<\min \left\{
1,\varepsilon \right\} ).$ Choose $n$ and $\delta $ so that $n>\frac{100}{%
\varepsilon }$ and $2^{-\left( n+1\right) }<\frac{\varepsilon ^{\prime }}{100%
}$ and $\delta <\frac{\left( \varepsilon ^{\prime }\right) ^{4}}{2^{n}\left(
100\right) }.$ The number $\delta $\ is chosen in part so that the domain of 
$\bar{S}^{n}$ has measure greater than $1-\frac{\varepsilon }{100}$ and also
so that the union of the atoms of measure less than $\delta $\ in any
distribution with $\leq 2^{n}$\ atoms has measure less than $\frac{\left(
\varepsilon ^{\prime }\right) ^{4}}{100}.$ Additional features of the
dependence of $\delta $\ on $\varepsilon ^{\prime }$\ and $n$\ will be given
below. Suppose the above hypotheses are met concerning the given $G-$%
extensions and the speedup $\bar{S}.$ Let $\bar{R}$ denote the speedup tower
for $\bar{S}.$ We may assume, without loss of generality, that $\left( \bar{%
\mu}\times \lambda \right) \left( \bar{R}\right) <1-\frac{\delta }{2}.$

Fix a $t-$element partition $Q$ of $\left( P\times G\right) ^{\left[ n\right]
}$ whose atoms are continuity sets of diameter less than $\frac{\delta }{100}
$ (in the \textquotedblleft max\textquotedblright\ metric $\rho ^{\prime }$
using the discrete metric on $P$) for the measure $dist_{X\times
G}\bigvee_{i\in \left[ n\right] }S^{-i}\left( P\vee c\right) .$ Thus, for
all atoms $q\in Q$, and all $x,y\in q$ the $P-n$ names of $x$ and $y$ are
equal, and the $c-n$ names are uniformly close.

Let $n_{1}\in \mathbb{N}$ and $\delta _{1}>0$ be given. We may assume that $%
\delta _{1}<\delta ,$ and we will have occasion to replace $\delta _{1}$ by
an even smaller number during the argument. Fix a measurable set $%
A_{1}\subset \bar{X}$, and let $A=A_{1}\times A_{2}$.

Fix $\zeta >0,$ whose size will be determined by what follows. Choose $K\in 
\mathbb{N}$ by lemma $\ref{sampling}$ with respect to $n,\delta $ and $\zeta 
$.

Using lemma $\ref{model name 2}$ we fix a model name $F$ for the
\textquotedblleft target\textquotedblright\ process $\left( S,P\vee c\right)
,$ so that (replacing $n_{1}$ by a larger number if necessary, which we
still call $n_{1}$) $n_{1}>n/\zeta $, $\left\vert F\right\vert >n_{1}/\zeta $
and\newline
(a.)%
\begin{equation*}
\left\Vert dist_{\left[ \left\vert F\right\vert -n_{1}+1\right] }F_{n_{1}},\
dist_{X\times G}\bigvee_{i\in \left[ n_{1}\right] }S^{-i}\left( P\vee
c\right) \right\Vert _{\mathcal{M}}<\frac{\delta _{1}}{100}
\end{equation*}%
(b.)%
\begin{equation*}
\left\Vert dist_{\left[ \frac{\left\vert F\right\vert }{n_{1}}\right]
n_{1}}F_{n_{1}},\ dist_{X\times G}\bigvee_{i\in \left[ n_{1}\right]
}S^{-i}\left( P\vee c\right) \right\Vert _{\mathcal{M}}<\frac{\delta _{1}}{%
100}
\end{equation*}%
(c.) For every $k\in \left[ \frac{\left\vert F\right\vert }{n_{1}}\right] ,$ 
$H_{k}:=F_{n_{1}}\left( kn_{1}\right) \ $is $\left( 1-\zeta \right) -$%
covered by disjoint $n-$blocks, such that if $\mathcal{J}_{k}$ is the set of
their initial positions, 
\begin{equation*}
\left\Vert dist_{\mathcal{J}_{k}}F_{n},\ dist_{X\times G}\bigvee_{i\in \left[
n\right] }S^{-i}\left( P\vee c\right) \right\Vert _{\mathcal{M}}<\tilde{\zeta%
},
\end{equation*}%
where $\tilde{\zeta}$ is determined by lemma \ref{continuity sets} with
respect to $\zeta $.\newline
(d.) For every $k,$ and for every $q\in Q$, $\left\vert \left\{ j\in 
\mathcal{J}_{k}\mid Q\left( F_{n}\left( j\right) \right) =q\right\}
\right\vert \geq K.$

We may assume further that for all $k$ and $k^{\prime },$ $\left\vert 
\mathcal{J}_{k}\right\vert =\left\vert \mathcal{J}_{k^{\prime }}\right\vert
. $ For each $k>0,$ we fix a bijection $\psi _{k}:\mathcal{J}_{k}\rightarrow 
\mathcal{J}_{0}$ such that $\left\Vert dist_{\mathcal{J}_{k}}F_{n},dist_{%
\mathcal{J}_{k}}F_{n}\psi _{k}\right\Vert _{\mathcal{M}}<\zeta .$ (To be
precise, we must have originally chosen $\tilde{\zeta}$ as in lemma \ref%
{dist match 1}$,$ so that such $\zeta -$distribution matches are available
here). We let $\psi _{0}:\mathcal{J}_{0}\rightarrow \mathcal{J}_{0}$ be the
identity.

For each $k,$ and each $J\in \mathcal{J}_{k},$ we let $b\left( J\right) $
denote the $n-$block with initial position $J,$ and refer to $b\left(
J\right) $ as a \emph{real} $n-$block of $H_{k}.$ Each component of $%
H_{k}\backslash \left( \cup _{J\in \mathcal{J}_{k}}b\left( J\right) \right) $
is further partitioned into blocks of length no greater than $n,$ using as
many blocks of length $n$ as possible. We refer to these blocks as \emph{%
pseudo} $n-$blocks. We let $\mathcal{J}_{k}^{\prime }=\left\{
J_{k,i}\right\} _{i}$ denote the set of all initial positions of $n-$blocks
(real and pseudo) in $H_{k},$ listed in order. We also write $\left\{
J_{k,i_{m}}\right\} $ (respectively $\left\{ J_{k,j_{m}}\right\} )$ for the
initial positions of the real (respectively pseudo) $n-$blocks in $H_{k},$
listed in order$.$

We note that as a consequence of (c.), for every $k$ and $k^{\prime },$ the
distributions of $Q$ on the disjoint $n-$blocks of $H_{k}$ and $H_{k^{\prime
}}$ are $2\zeta -$close to each other.

The need to replace $n_{1}$ with a larger number is the reason that the
statement of the lemma says \textquotedblleft for all sufficiently large $%
n_{1}$\textquotedblright .

To support our construction, we use a pair of Rokhlin towers. First, we
construct a Rokhlin tower $R$ for $\bar{S}_{0}$, with base $B\in \mathcal{%
\bar{A}}$ and height $M$ so that $\left( \bar{\mu}\times \lambda \right)
\left( R\right) >1-\zeta ,$ and for all $y\in B,$ $\bar{S}_{0}^{\left[ M%
\right] }\left( y\right) $ admits a disjoint collection of ladder blocks of $%
\bar{S},$ such that, if $\mathcal{V}_{y}$ denotes the set of initial
elements of these $\bar{S}-n-$blocks,%
\begin{equation}
\left\vert \frac{n\left\vert \mathcal{V}_{y}\right\vert }{M}-\left( \bar{\mu}%
\times \lambda \right) \left( \bar{R}\right) \right\vert <\zeta
\label{cover M-name}
\end{equation}%
\begin{equation}
\left\Vert dist_{\mathcal{V}_{y}}\bigvee_{i\in \left[ n\right] }\bar{S}%
^{-i}\left( \bar{P}\vee \bar{c}\right) ,dist_{X\times G}\bigvee_{i\in \left[
n\right] }S^{-i}\left( P\vee c\right) \right\Vert _{\mathcal{M}}<\delta
\label{good n-dist}
\end{equation}%
and for all $y^{\prime }\in B$%
\begin{equation}
\left\Vert dist_{\mathcal{V}_{y}}\bigvee_{i\in \left[ n\right] }\bar{S}%
^{-i}\left( \bar{P}\vee \bar{c}\vee \boldsymbol{1}_{A_{1}}\right) ,dist_{%
\mathcal{V}_{y^{\prime }}}\bigvee_{i\in \left[ n\right] }\bar{S}^{-i}\left( 
\bar{P}\vee \bar{c}\vee \boldsymbol{1}_{A_{1}}\right) \right\Vert _{\mathcal{%
M}}<\zeta  \label{same n-dist}
\end{equation}%
and for all $y,y^{\prime }\in B$%
\begin{equation}
\left\vert \mathcal{V}_{y}\right\vert =\left\vert \mathcal{V}_{y^{\prime
}}\right\vert .  \label{same number}
\end{equation}%
To construct $R$ let $\bar{B}$ denote the base of $\bar{R}$ and $\bar{L}$
its height, and let $\bar{k}$ denote the variable exponent that gives $\bar{S%
}=\bar{S}_{0}^{\bar{k}}.$ For each $M\in \mathbb{N}$ and $y\in \bar{X}\times
G$ we let $\mathcal{R}_{y}=\left\{ y^{\prime }\in \bar{S}_{0}^{\left[ M%
\right] }\left( y\right) \mid \bar{S}^{\mathbb{Z}}\left( y^{\prime }\right)
\subset \bar{S}_{0}^{\left[ M\right] }\left( y\right) \right\} $ and $%
\mathcal{Z}_{y}=\bar{B}\cap \mathcal{R}_{y}.$ Consider the set $\tilde{Y}%
\subset \bar{X}\times G$ consisting of those $y$ such that%
\begin{equation*}
\left\vert \frac{\left\vert \mathcal{R}_{y}\right\vert }{M}-\left( \bar{\mu}%
\times \lambda \right) \left( \bar{R}\right) \right\vert <\zeta
\end{equation*}%
and%
\begin{equation*}
\left\Vert dist_{\mathcal{Z}_{y}}\bigvee_{i\in \left[ \bar{L}\right] }\bar{S}%
^{-i}\left( \bar{P}\vee \bar{c}\vee \boldsymbol{1}_{A_{1}}\right) ,dist_{%
\bar{B}}\bigvee_{i\in \left[ \bar{L}\right] }\bar{S}^{-i}\left( \bar{P}\vee 
\bar{c}\vee \boldsymbol{1}_{A_{1}}\right) \right\Vert _{\mathcal{M}}<\frac{%
\zeta }{100}.
\end{equation*}%
By the ergodic theorem, and using the fact that $\bar{k}$ is bounded, we
know that if $M$ is sufficiently large, then $\left( \bar{\mu}\times \lambda
\right) \left( \tilde{Y}\right) >1-\frac{\zeta }{100}.$ Suppose $y\in \tilde{%
Y}.$ Then for each $y^{\prime }\in \mathcal{Z}_{y}$ the set $\bar{S}^{\left[
L\right] }\left( y^{\prime }\right) $ is divided into ladder blocks (of
length $n$), and if $\mathcal{V}_{y}$ denotes the initial elements of all
these blocks, conditions $\left( \ref{cover M-name}\right) ,\left( \ref{good
n-dist}\right) ,$ and $\left( \ref{same n-dist}\right) $ are satisfied.

The set $\tilde{Y}$ is $\mathcal{\bar{A}}_{0}-$measurable, so lemma \ref%
{strR} applied to $\bar{T}_{0}$ gives a Rokhlin tower $R$ for $\bar{S}_{0}$
with base $B\subset \tilde{Y}$ that satisfies the desired conditions. By
deleting some members of the sets $\mathcal{V}_{y},$ we can also arrange
that $\left( \ref{same number}\right) $ holds.

In the construction of the tower $R$ we also choose $M$ so that any
distribution match as in lemma \ref{dist match 2} between $dist_{\mathcal{V}%
_{y}}\left( \bigvee_{i\in \left[ n\right] }\bar{S}^{-i}\left( \bar{P}\vee 
\bar{c}\vee \boldsymbol{1}_{A_{1}}\right) \right) $ for a point $y\in B$ and 
$dist_{\mathcal{J}_{k}}F_{n}$ will be at least $\frac{1}{\zeta }-$to-one. We
will refer to the $\bar{S}-n-$blocks with initial points in the sets $%
\mathcal{V}_{y}$ as \emph{useful }blocks.

Second, we construct a much longer Rokhlin tower $R^{\prime }$ with height $%
M^{\prime }$ and base $B^{\prime }\in \mathcal{\bar{A}}_{0}\ $such that $%
\left( \bar{\mu}\times \lambda \right) \left( R^{\prime }\right) >1-\zeta $
and so that for each $y\in B^{\prime },$ $\bar{S}_{0}^{\left[ M^{\prime }%
\right] }y$ is $\left( 1-\zeta \right) -$covered by (necessarily disjoint)
sets of the form $\bar{S}_{0}^{\left[ M\right] }y^{\prime }$ where $%
y^{\prime }\in B.$ The height $M^{\prime }$ will be chosen subject to some
additional requirements, which will be described below. We divide $R^{\prime
}$ into columns whose levels are pure with respect to $\bar{P},\boldsymbol{1}%
_{A_{1}},$ the variable exponent $\bar{k},$ the levels of the speedup tower $%
\bar{R}$ for $\bar{S},$ and the levels of $R.$ We further refine the columns
of $R^{\prime }$ so that in each level $L$ of a column, the values of the
skewing function $\bar{\sigma}_{0}$ are within $\zeta ^{\prime }$ of being
constant, where $\zeta ^{\prime }$ is so small that if $O_{\zeta ^{\prime
}}(id_{G})$ denotes the $\zeta ^{\prime }-$ball in $G$ centered at $id_{G},$
then 
\begin{equation}
\left( O_{\zeta ^{\prime }}(id_{G})\right) ^{M^{\prime }}\subset O_{\zeta
}(id_{G}).  \label{small zeta prime}
\end{equation}

Fix a column $C^{\prime }$ in $R^{\prime }$ and fix $y^{\prime }\in
B^{\prime }\cap C^{\prime }.$ If $j\in \left[ M^{\prime }-M\right] $ and $%
\bar{S}_{0}^{j}y^{\prime }\in B$ then we refer to $\bar{S}_{0}^{\left[ M%
\right] }\left( \bar{S}_{0}^{j}y^{\prime }\right) $ as a \emph{window} (for $%
y^{\prime }).$ (This language anticipates the construction of cycles below).
We let $\left\{ W_{i}\right\} _{i=0}^{w-1}$ be the set of windows for $%
y^{\prime }$, indexed in the order imposed by $\bar{S}_{0}.$ We recall that
each such window is covered, up to a fraction $\left( \bar{\mu}\times
\lambda \right) \left( \bar{R}\right) \pm \zeta ,$ by the useful blocks
associated with $\bar{S}_{0}^{j}\left( y\right) $. We note that the useful
blocks in each $W_{i}$ occupy no more than a $\left( 1-\frac{\delta }{2}%
+\zeta \right) -$fraction of $W_{i}$. We let $\mathcal{V}_{i}$ denote the
set of initial points of these useful blocks.

For each $i>0$ we fix a bijection $\phi _{i}:\mathcal{V}_{i}\rightarrow 
\mathcal{V}_{0}$ that is a $\zeta -$distribution match between $dist_{%
\mathcal{V}_{i}}\bigvee_{i\in \left[ n\right] }\bar{S}^{-i}\left( \bar{P}%
\vee \bar{c}\vee \boldsymbol{1}_{A_{1}}\right) $ and $dist_{\mathcal{V}%
_{0}}\bigvee_{i\in \left[ n\right] }\bar{S}^{-i}\left( \bar{P}\vee \bar{c}%
\vee \boldsymbol{1}_{A_{1}}\right) .$ (As before, we must have originally
chosen a number $\tilde{\zeta}$ as in lemma \ref{dist match 1} instead of $%
\zeta ,$ so that such $\zeta -$distribution matches are available here).

Let $\theta :\mathcal{V}_{0}\rightarrow \mathcal{J}_{0}$ be a $\delta -$%
distribution match between $dist_{\mathcal{V}_{0}}\bigvee_{i\in \left[ n%
\right] }\bar{S}^{-i}\left( \bar{P}\vee \bar{c}\right) $ and $dist_{\mathcal{%
J}_{0}}F_{n}$. (Again, we must have chosen a number $\tilde{\delta}=\tilde{%
\delta}\left( \delta ,n\right) $ as in lemma \ref{dist match 2} instead of $%
\delta $ in condition \ref{delta dist match}, so that such a $\delta -$%
distribution match is available here).

The desired partial speedup $\bar{S}_{1}=\bar{S}_{0}^{\bar{k}_{1}}$ of $\bar{%
S}_{0}$ will be defined first on $\bar{S}_{0}^{\left[ M^{\prime }\right]
}\left( y^{\prime }\right) $ by concatenating blocks taken from successive
windows in $\bar{S}_{0}^{\left[ M^{\prime }\right] }\left( y^{\prime
}\right) .$ These blocks will be selected so that the orbit they form will
be well matched to the model name $F.$ The definition of $\bar{S}_{1}$ will
then be extended to the rest of $C^{\prime }$ by making $\bar{k}_{1}$
constant on the levels of $C^{\prime }.$ All other columns will be treated
in a similar way. The details will be presented in a sequence of steps.

Step 1. Let $Q_{0}$ denote the partition $Q\circ \theta $ on $\mathcal{V}%
_{0}.$ Let $\mathcal{R}_{0}$ denote the partition of $\mathcal{V}_{0}$ by $%
\bigvee_{i\in \left[ n\right] }\bar{S}^{-i}\left( \boldsymbol{1}%
_{A_{1}}\right) .$ We make a modification $\widetilde{\mathcal{R}}_{0}$ of $%
\mathcal{R}_{0}:$

Let $\nu $ denote the normalized counting measure on $\mathcal{V}_{0}$. For
each atom $q$ of $Q_{0},$ let $\nu _{q}$ denote $\nu $ conditioned on $q,$
and let $\mathcal{R}_{0}^{q}$ denote the restriction of $\mathcal{R}_{0}$ to 
$q.$ We will construct a new partition $\widetilde{\mathcal{R}}_{0}^{q}$ on $%
q,$ such that $\left\Vert dist_{q}\mathcal{R}_{0}^{q},dist_{q}\widetilde{%
\mathcal{R}}_{0}^{q}\right\Vert _{\mathcal{M}}<\frac{\varepsilon ^{\prime }}{%
50},$ and such that for every atom $\widetilde{r}$ of $\widetilde{\mathcal{R}%
}_{0}^{q}$, $\nu _{q}\left( \widetilde{r}\right) >\delta .$

Let $U_{q}=\tbigcup \left\{ a\in \mathcal{R}_{0}^{q}:\nu _{q}\left( a\right)
<\delta \right\} .$ By the choice of $\delta ,$ $\nu _{q}\left( U\right) <%
\frac{\varepsilon ^{\prime }}{100}.$ If $\nu _{q}\left( U_{q}\right) \geq
\delta $ we regard $U_{q}$ as a single atom of $\widetilde{\mathcal{R}}%
_{0}^{q},$ and we let $\widetilde{\mathcal{R}}_{0}^{q}$ coincide with $%
\mathcal{R}_{0}^{q}$ on the rest of its atoms. If $\nu _{q}\left(
U_{q}\right) <\delta $ we choose an atom $r$ of $\mathcal{R}_{0}^{q}$ such
that $\nu _{q}\left( r\right) >2^{-n}$ and a subset $r^{\prime }\subset r$
with $\nu _{q}\left( r^{\prime }\right) \in \left( 2^{-\left( n+2\right)
},2^{-\left( n+1\right) }\right) .$ We regard $r^{\prime }\cup U_{q}$ and $%
r\backslash r^{\prime }$ as single atoms of $\widetilde{\mathcal{R}}%
_{0}^{q}, $ and we let $\widetilde{\mathcal{R}}_{0}^{q}$ coincide with $%
\mathcal{R}_{0}^{q}$ on the rest of its atoms. (The $\nu _{q}-$measure of a
singleton is small enough to guarantee the existence of $r^{\prime }$).

The partition $\widetilde{\mathcal{R}}_{0}^{q}$ has the desired properties.
In particular, $\widetilde{\mathcal{R}}_{0}^{q}$ coincides with $\mathcal{R}%
_{0}^{q}$ except on the atom $r_{U_{q}}$ of $\widetilde{\mathcal{R}}_{0}^{q}$
that contains $U_{q}.$

We let $\widetilde{\mathcal{R}}_{0}$ denote the partition of $\mathcal{V}%
_{0} $ whose restriction to each atom $q$ is $\widetilde{\mathcal{R}}%
_{0}^{q}.$ Thus $\widetilde{\mathcal{R}}_{0}$ coincides with $\mathcal{R}%
_{0} $ except on the union of the sets $r_{U_{q}}.$ We will refer to the
atoms $r_{U_{q}}$ as \textquotedblleft miscellaneous\textquotedblright\
atoms.

Step 2. To each real $n-$block of $F$ we assign an atom of $Q_{0}\vee 
\widetilde{\mathcal{R}}_{0}:$ For each $q\in Q,$ let $\mathcal{J}%
_{0,q}=\left\{ j\in \mathcal{J}_{0}\mid Q\left( F_{n}\left( j\right) \right)
=q\right\} .$ Let $f_{q,0}:\mathcal{J}_{0,q}\rightarrow \widetilde{\mathcal{R%
}}_{0}^{q}$ be a function with statistical distribution within $\zeta $ of $%
dist_{q}\widetilde{\mathcal{R}}_{0}^{q}.$ Lemma \ref{sampling}\ guarantees
that such a function exists. Let $f_{0}:\mathcal{J}_{0}\rightarrow Q_{0}\vee 
\widetilde{\mathcal{R}}_{0}$ be the common extension\ of the $f_{q,0}.$That
is, $f_{0}=\cup _{q}f_{q,0}.$ For each $k>0$, we set $f_{k}=f\circ \psi _{k}:%
\mathcal{J}_{k}\rightarrow Q_{0}\vee \widetilde{\mathcal{R}}_{0},$ and we
set $f=\cup _{k}f_{k}$.

Step 3. We prepare samples from each $\mathcal{V}_{i}:$ Let $\left\{ \tau
_{0,t}:\cup _{k}\mathcal{J}_{k}\rightarrow \mathcal{V}_{0}\right\} _{t}$ be
injections with disjoint ranges so that for all $t$ and all $i\in \cup _{k}%
\mathcal{J}_{k},$ $Q_{0}\vee \widetilde{\mathcal{R}}_{0}^{q}\left( \tau
_{0,t}\left( i\right) \right) =f\left( i\right) .$ To see that a large
collection of such \textquotedblleft samples\textquotedblright\ $\tau _{0,t}$
is available, let $\delta ^{\prime \prime }=\frac{1}{2}\min_{q\in Q}\left\{
\mu \times \lambda \left( q\right) \right\} $ and $\delta ^{\prime }=\delta
^{\prime \prime }\delta .$ If $\zeta $ was chosen so that $\zeta <\min
\left\{ \frac{\delta ^{\prime \prime }}{100},\frac{\delta _{1}\delta
^{\prime }}{200}\right\} ,$ then by property $\left( c.\right) $ in our
choice of $F,$ it will be the case that $\delta ^{\prime \prime }<\min_{q\in
Q_{0}}\nu \left( q\right) $, and so $\delta ^{\prime }<\min_{a\in Q_{0}\vee 
\mathcal{\tilde{R}}_{0}}\nu \left( a\right) $. Let $N^{\prime }$ be the
number given by lemma \ref{exhaustion} with respect to $\delta ^{\prime },%
\frac{\delta _{1}}{100},\zeta $ and $K^{\prime }=\left\vert \mathcal{\cup }%
_{k}\mathcal{J}_{k}\right\vert .$ Then the tower $R$ could have been chosen
so that the number of useful blocks in each window exceeds $N^{\prime }.$
Applying lemma \ref{exhaustion} we obtain a set of samples $\left\{ \tau
_{0,t}\right\} _{t}$ such that $\cup _{t}\tau _{0,t}\left( \cup _{k}\mathcal{%
J}_{k}\right) $ covers all but a $\frac{\delta _{1}}{100}-$fraction of $%
\mathcal{V}_{0}$.

For each $s>0,$ we set $\tau _{s,t}=\phi _{s}^{-1}\tau _{0,t}:\cup _{k}%
\mathcal{J}_{k}\rightarrow \mathcal{V}_{s},$ which provides corresponding
samples of $\mathcal{V}_{s}.$

Step 4. We construct cycles in $\left[ M^{\prime }\right] :$ Identifying $%
\bar{S}_{0}^{\left[ M^{\prime }\right] }y^{\prime }$ with $\left[ M^{\prime }%
\right] $ via $e:\bar{S}_{0}^{i}y^{\prime }\mapsto i,$ we have a system $%
\left\{ \tilde{W}_{s}\right\} $ of $w$ windows of length $M$ in $\left[
M^{\prime }\right] .$ Let $p=\left\vert \mathcal{J}^{\prime }\right\vert $
be the number of $n-$blocks in $F$. We also identify $\mathcal{J}^{\prime }$
with $\left[ p\right] $ (without changing notation)$,$ and for each $t,$ we
construct a cycle $\Gamma _{t}$ of $p-$sequences in $\left[ M^{\prime }%
\right] $ as follows. Let $\tau _{s,t}^{\prime }\left( i\right) $ denote the
height of the point $\tau _{s,t}\left( i\right) $ above the base of $W_{s}.$
We (partially) define $\Gamma _{t}=\left\{ g_{j}^{l,t}\right\} _{l,j}$ by
setting, 
\begin{equation*}
g_{j}^{l,t}\left( i\right) =\tilde{W}_{jp+l+i}\left( \tau
_{jp+l+i,t}^{\prime }\left( i\right) \right) .
\end{equation*}

For the moment, we leave the functions $g_{j}^{l,t}$ of $\Gamma _{t}$
undefined on $\mathcal{J}^{\prime }\backslash \cup _{k}\mathcal{J}_{k}.$ We
will refer to these $4-$tuples $\left( t,l,j,i\right) $ as the \emph{real }$%
4-$tuples (associated with $C^{\prime })$. We have that $\left\{
g_{j}^{l,t}\left( i\right) \mid \left( t,l,j,i\right) \text{ is a real }4-%
\text{tuple}\right\} $ covers all but a $\frac{\delta _{1}}{100}-$fraction
of $e\left( \cup _{s}\mathcal{V}_{s}\right) .$

Step 5. We extend the domains of the\ functions that make up the cycles $%
\Gamma _{t}$ to all of $\left[ p\right] :$ For each window $W_{s}$ and each $%
z\in \mathcal{V}_{s}$ let $b\left( z\right) $ denote the useful block
beginning at $z.$ The set $U_{s}=W_{s}\backslash \left[ \cup
_{t,l,j,i}b\left( e^{-1}g_{j}^{l,t}\left( i\right) \right) \right] $ covers
at least a $\left( \frac{\delta }{2}-\zeta \right) -$fraction of $W_{s}$.
For each $4-$tuple $\left( t,l,j,i\right) $ where $\left( t,l,j\right) $\ is
the initial triple of a real $4-$tuple and $i$ is initial position of a
pseudo $n-$block of length $\tilde{n}$, we choose a subset $\gamma \left(
\left( t,l,j,i\right) \right) \subset U_{s}$ of size $\tilde{n}.$ The sets $%
\gamma \left( \left( t,l,j,i\right) \right) $ are chosen to be pairwise
disjoint. Since the number of pseudo $n-$blocks in $F$ is at most a $\zeta -$%
fraction of the number of real $n-$blocks, and $\zeta $ is much less than $%
\delta ,$ such disjoint sets are available. We define $g_{j}^{l,t}\left(
i\right) $ to be the first element of $e\left( \gamma \left( \left(
t,l,j,i\right) \right) \right) .$

Step 6. We begin to define $\bar{S}_{1}:$ For each stage $g_{j}^{l,t}$ in
one of the constructed cycles let $\tilde{g}_{j}^{l,t}=e^{-1}g_{j}^{l,t}.$
Each $i^{\prime }\in \mathcal{J}^{\prime }$ is the initial position of an $n$
block (real or pseudo) of length $\tilde{n},$ and $\tilde{g}_{j}^{l,t}\left(
i^{\prime }\right) $ is the initial point of a block of the same length $%
\tilde{n}$ in a single window of $\bar{S}_{0}^{\left[ M^{\prime }\right]
}\left( y^{\prime }\right) .$ That block is either the useful block
beginning at $\tilde{g}_{j}^{l,t}\left( i^{\prime }\right) $ or a block of
the form $\gamma \left( \left( t,l,j,i\right) \right) $ as in step $5$. The
blocks with initial points $\tilde{g}_{j}^{l,t}\left( i^{\prime }\right) $
lie in successive windows, and we extend the map $\tilde{g}_{j}^{l,t}$ to a
map $\hat{g}_{j}^{l,t}:\left[ \left\vert F\right\vert \right] \rightarrow 
\bar{S}_{0}^{\left[ M^{\prime }\right] }\left( y^{\prime }\right) $ by
concatenating these blocks. It follows that $e\hat{g}_{j}^{l,t}:\left[
\left\vert F\right\vert \right] \rightarrow \left[ M^{\prime }\right] $ is
increasing. We regard the image of $\hat{g}_{j}^{l,t}$ as an orbit of the
new partial speedup $\bar{S}_{1}.$ Writing $\bar{S}_{1}=\bar{S}_{0}^{\bar{k}%
_{1}}$ along this orbit, we extend the definition of $\bar{S}_{1}$ to the
union of the levels of $C^{\prime }$ that contain this orbit, by requiring
that $\bar{k}_{1}$ be constant on each of these levels. We refer to this
union of levels as a speedup column.

Step 7. We define the new partition $\bar{P}$ and the adjustment function $%
\bar{\alpha}$ on each of the speedup columns just created: Our model name $F$
has the form $F=\left\{ \left( p_{s},g_{s}^{\prime }\right) \right\} _{s\in %
\left[ \left\vert F\right\vert \right] }\in \left( P\times G\right) ^{\left[
\left\vert F\right\vert \right] }.$ In this speedup column we have the
particular speedup orbit, taken from the orbit of $y^{\prime },$ which is a
sequence $\left\{ \left( \bar{x}_{s},g_{s}\right) \right\} _{s\in \left[
\left\vert F\right\vert \right] }$ of points in $\bar{X}\times G.$ We define 
$\bar{P}_{1}$ on the points $\left\{ \bar{x}_{s}\right\} _{s\in \left[
\left\vert F\right\vert \right] }$ by setting, for each $s,$ $\bar{P}%
_{1}\left( \bar{x}_{s}\right) =p_{s}$, and we make $\bar{P}_{1}$ constant on
the column level that contains $\bar{x}_{s}.$ (We can view $\bar{P}$ as
either a partition of $\bar{X}$ or of $\bar{X}\times G$ without confusion).
We define $\bar{\alpha}$ on the points $\left\{ \bar{x}_{s}\right\} _{s\in %
\left[ \left\vert F\right\vert \right] }$ by requiring, for each $s,$ $\bar{%
\alpha}\left( \bar{x}_{s}\right) g_{s}=g_{s}^{\prime }.$ We extend $\bar{%
\alpha}$ to the whole speedup column we have constructed as follows. Let $L$
denote the base of this column. For each point $\bar{z}\in \bar{X}$ such
that $\left\{ \bar{z}\right\} \times G\subset L,$ we consider the orbit of $%
\left( \bar{z},g_{0}\right) $ under the speedup $\bar{S}_{1}$ that we have
defined (so far just on this column). The $G-$ coordinates of the points on
this orbit are rotations of $g_{0}$ by successive products of the skewing
function $\bar{\sigma}$ associated with $\bar{S}_{1}.$ Namely, 
\begin{equation*}
c\left( \bar{S}_{1}^{t}\left( \bar{z},g_{0}\right) \right) =\left( \bar{T}%
_{1}\left( \bar{z}\right) ,\bar{\sigma}^{\left( t\right) }\left( \bar{z}%
\right) g_{0}\right)
\end{equation*}%
where $\bar{\sigma}^{\left( t\right) }\left( \bar{z}\right) =\bar{\sigma}%
\left( \bar{T}_{1}^{t}\left( \bar{z}\right) \right) ...\bar{\sigma}\left( 
\bar{T}_{1}^{2}\left( \bar{z}\right) \right) \bar{\sigma}\left( \bar{T}%
_{1}\left( \bar{z}\right) \right) \bar{\sigma}\left( \bar{z}\right) .$ But
for each $j\in \left[ t\right] ,$ 
\begin{equation*}
\rho \left( \bar{\sigma}\left( \bar{T}_{1}^{j}\left( \bar{z}\right) \right) ,%
\bar{\sigma}\left( \bar{T}_{1}^{j}\left( \bar{x}_{0}\right) \right) \right)
<\zeta ^{\prime }
\end{equation*}%
and $t\leq M^{\prime },$ so by the choice of $\zeta ^{\prime }$ (see (\ref%
{small zeta prime})), we get $\rho \left( \bar{\sigma}^{\left( t\right)
}\left( \bar{z}\right) ,\bar{\sigma}^{\left( t\right) }\left( \bar{x}%
_{0}\right) \right) <\zeta .$ We define $\bar{\alpha}$ on the $\bar{T}_{1}$
orbit of $\bar{z}$ by requiring that 
\begin{equation*}
\bar{\alpha}\left( \bar{T}_{1}^{t}\left( \bar{z}\right) \right) \bar{\sigma}%
^{\left( t\right) }\left( \bar{z}\right) g_{0}=g_{t}^{\prime }
\end{equation*}%
In other words, the \textquotedblleft adjusted\textquotedblright\ $G-$%
coordinates of the points on this orbit are identical to the $G-$coordinates
of the corresponding terms in the model name $F.$

The useful blocks fill at least a $\left( 1-\delta \right) -$fraction of
each window, and we used a $\left( 1-\frac{\delta _{1}}{100}\right) -$
fraction of the useful blocks in each window. We may assume that $\delta
_{1} $ was chosen so that $\left( 1-\frac{\delta _{1}}{100}\right) \left(
1-\delta \right) >\left( 1-2\delta \right) .$ Since the windows $W_{i}$
(even those not within $p$ windows of the top and bottom) occupy at least a $%
\left( 1-\zeta \right) -$fraction of $C^{\prime },$ it follows (providing $%
\zeta $ is chosen small enough compared to $\delta _{1}$) that the speedup
columns for $\bar{S}_{1}$ constructed thus far cover at least a $\left(
1-3\delta \right) -$fraction of $C^{\prime }$.

Step 8. We repeat the preceding construction on each of the other columns of 
$R^{\prime }.$ Since $\bar{\mu}\times \lambda \left( R^{\prime }\right)
>1-\zeta ,$ this yields a speedup tower $\bar{R}_{0}$ for $\bar{S}_{1}$ with 
$\bar{\mu}\times \lambda \left( \bar{R}_{0}\right) >1-3\delta -\zeta .$

Step 9. We extend the definition of $\bar{S}_{1},\bar{P}_{1}$ and $\bar{%
\alpha}$ to more of $R^{\prime }:$ In each column of $R^{\prime },$ we
assemble the remaining levels in an order preserving way to form columns of
height $\left\vert F\right\vert ,$ making as many as the number of remaining
levels allows. This results in a speedup tower $\bar{R}_{1}$ for $\bar{S}%
_{1} $ whose base $B_{1}$ contains $B_{0},$ and we have $\bar{\mu}\times
\lambda \left( \bar{R}_{1}\right) >1-2\zeta \ $and $\bar{\mu}\times \lambda
\left( dom\left( \bar{S}_{1}\right) \right) >1-3\zeta >1-\delta _{1}$. On
each of the columns we just constructed, we define $\bar{P}_{1}$ and $\bar{%
\alpha}$ as before to match the name $F.$ We note that the partition $\bar{P}%
_{1}$ and the function $\bar{\alpha}$ have been defined on $\bar{R}_{1}$ so
that for every point $z$ in the base $B_{1},$ whose $G-$coordinate equals
the $G-$coordinate of the first term in $F$, 
\begin{equation*}
\bigvee_{i\in \left[ \left\vert F\right\vert \right] }\bar{S}_{1}^{-i}(\bar{P%
}_{1}\vee \bar{\alpha}c)\left( z\right) =F.
\end{equation*}%
We also extend the definition of $\bar{P}_{1}$ and $\bar{\alpha}_{1}$ to the
complement of $\bar{R}_{1}$, by including the complement of $\bar{R}_{1}$ in
a single atom of $\bar{P}_{1}$ and by setting $\bar{\alpha}_{1}$ equal to $%
id_{G}$ on the complement of $\bar{R}_{1}.$

We now verify that the conclusions of the lemma hold.

To establish (\ref{n_1,d_1 regular}), the only property requiring
explanation is the fourth of the definition of $\left( n_{1},\delta
_{1}\right) -$regularity. But for all $z$ as in step 9, properties (a.) and
(b.) in the formation of $F$ imply that, with $\mathcal{W}_{z}=\bar{S}_{1}^{%
\left[ \frac{\left\vert F\right\vert }{n_{1}}\right] n_{1}}\left( z\right) ,$%
\begin{equation*}
\left\Vert dist_{\mathcal{W}_{z}}\bigvee_{i\in \left[ n_{1}\right] }\bar{S}%
_{1}^{-i}\left( \bar{P}_{1}\vee \bar{\alpha}c\right) ,dist_{Dom\left( \bar{S}%
_{1}^{n_{1}}\right) }\bigvee_{i\in \left[ n_{1}\right] }\bar{S}%
_{1}^{-i}\left( \bar{P}\vee \bar{\alpha}c\right) \right\Vert _{\mathcal{M}}<%
\frac{\delta _{1}}{50}.
\end{equation*}%
And so for each $h\in G,$%
\begin{equation*}
\left\Vert dist_{\mathcal{W}_{zh}}\bigvee_{i\in \left[ n_{1}\right] }\bar{S}%
_{1}^{-i}\left( \bar{P}_{1}\vee \bar{\alpha}c\right) ,dist_{Dom\left( \bar{S}%
_{1}^{n_{1}}\right) }\bigvee_{i\in \left[ n_{1}\right] }\bar{S}%
_{1}^{-i}\left( \bar{P}\vee \bar{\alpha}c\right) \right\Vert _{\mathcal{M}}
\end{equation*}%
\begin{equation*}
=\left\Vert dist_{\mathcal{W}_{z}}\bigvee_{i\in \left[ n_{1}\right] }\bar{S}%
_{1}^{-i}\left( \bar{P}_{1}\vee \bar{\alpha}ch\right) ,dist_{Dom\left( \bar{S%
}_{1}^{n_{1}}\right) }\bigvee_{i\in \left[ n_{1}\right] }\bar{S}%
_{1}^{-i}\left( \bar{P}\vee \bar{\alpha}ch\right) \right\Vert _{\mathcal{M}}
\end{equation*}%
\begin{equation*}
=\left\Vert dist_{\mathcal{W}_{z}}\bigvee_{i\in \left[ n_{1}\right] }\bar{S}%
_{1}^{-i}\left( \bar{P}_{1}\vee \bar{\alpha}c\right) ,dist_{Dom\left( \bar{S}%
_{1}^{n_{1}}\right) }\bigvee_{i\in \left[ n_{1}\right] }\bar{S}%
_{1}^{-i}\left( \bar{P}\vee \bar{\alpha}c\right) \right\Vert _{\mathcal{M}}<%
\frac{\delta _{1}}{50}
\end{equation*}

Next we establish (\ref{close P}) and (\ref{small alpha}). Fix a column $%
C^{\prime }$ of $R^{\prime }.$ We will show that for most real $4-$tuples $%
\left( t,l,j,i\right) $ associated with $C^{\prime },$ 
\begin{equation}
\rho ^{\prime }\left[ F_{n}\left( i\right) ,\bigvee_{u\in \left[ n\right] }%
\bar{S}^{-u}\left( \bar{P}\vee \bar{c}\right) \left( \tilde{g}%
_{j}^{l,t}\left( i\right) \right) \right] <2\delta  \label{e close n names}
\end{equation}%
For brevity, we let $\psi =\cup _{k}\psi _{k}$.

Condition (\ref{e close n names}) is met if the following conditions on $%
\left( t,l,j,i\right) $ hold:

(a) $\rho ^{\prime }\left[ F_{n}\left( i\right) ,F_{n}\left( \psi \left(
i\right) \right) \right] <\zeta $,

(b) $\rho ^{\prime }\left[ F_{n}\left( \psi \left( i\right) \right)
,\bigvee_{u\in \left[ n\right] }\bar{S}^{-u}\left( \bar{P}\vee \bar{c}%
\right) \left( \tau _{0,t}\psi \left( i\right) \right) \right] <\delta +%
\frac{\delta }{100},$ and

(c) $\rho ^{\prime }\left[ \bigvee_{u\in \left[ n\right] }\bar{S}^{-u}\left( 
\bar{P}\vee \bar{c}\right) \left( \tau _{0,t}\psi \left( i\right) \right)
,\bigvee_{u\in \left[ n\right] }\bar{S}^{-u}\left( \bar{P}\vee \bar{c}%
\right) \left( \tau _{jp+l+i,t}\left( i\right) \right) \right] <\zeta .$

(Note that $\tau _{jp+l+i,t}\left( i\right) =\hat{g}_{j}^{l,t}\left(
i\right) $.)

Each of these conditions holds for a fraction of the set of real $4-$tuples
which is greater than

(a) $\left( 1-\zeta \right) $ (b)$\left( 1-\frac{\delta }{1-\frac{\delta _{1}%
}{100}}\right) $ (c) $\left( 1-2\zeta \right) .$ Concerning (b), we know
that for all $y$ in a set $\mathcal{V}_{0}^{\prime }\subset \mathcal{V}_{0}$
with $\frac{\left\vert \mathcal{V}_{0}^{\prime }\right\vert }{\left\vert 
\mathcal{V}_{0}\right\vert }>\left( 1-\delta \right) $ we have $\rho
^{\prime }\left[ \bigvee_{u\in \left[ n\right] }\bar{S}^{-u}\left( \bar{P}%
\vee \bar{c}\right) \left( y\right) ,F_{n}\left( \theta y\right) \right]
<\delta ,$ so if $\tau _{0,t}\left( \psi \left( i\right) \right) =y\in 
\mathcal{V}_{0}^{\prime }$ then $\rho ^{\prime }\left[ F_{n}\left( \psi
\left( i\right) \right) ,\bigvee_{u\in \left[ n\right] }\bar{S}^{-u}\left( 
\bar{P}\vee \bar{c}\right) \left( \tau _{0,t}\psi \left( i\right) \right) %
\right] \leq $ $\rho ^{\prime }\left[ F_{n}\left( \psi \left( i\right)
\right) ,F_{n}\left( \theta y\right) \right] $ $+\rho ^{\prime }\left[
F_{n}\left( \theta y\right) ,\bigvee_{u\in \left[ n\right] }\bar{S}%
^{-u}\left( \bar{P}\vee \bar{c}\right) \left( \tau _{0,t}\psi \left(
i\right) \right) \right] <\frac{\delta }{100}+\delta .$ Here the fact that $%
Q\left( F_{n}\left( \psi \left( i\right) \right) \right) =Q\left(
F_{n}\left( \theta y\right) \right) $ implies $\rho ^{\prime }\left[
F_{n}\left( \psi \left( i\right) \right) ,F_{n}\left( \theta y\right) \right]
<\frac{\delta }{100}.$ Since $\left\{ \tau _{0,t}\left( \psi \left( i\right)
\right) \right\} _{t,i}$ covers a $\left( 1-\frac{\delta _{1}}{100}\right) -$%
fraction of $\mathcal{V}_{0},$ we get condition (b).

Thus, (if $\delta _{1}$ and $\zeta $ are sufficiently small with respect to $%
\delta $) for a set of $4-$tuples of density greater than $\left( 1-2\delta
\right) $, we have condition (\ref{e close n names}).

But the blocks that arise from these $4-$tuples (that is the useful blocks
with initial points $\tilde{g}_{j}^{l,t}\left( i\right) )$ occupy at least a 
$\left( 1-\frac{\delta _{1}}{100}\right) -$fraction of the set of all useful
blocks in $S^{\left[ M^{\prime }\right] }y^{\prime },$ and the useful blocks
in $C^{\prime }$ occupy at least a $\left( 1-\delta -\zeta \right) -$%
fraction of the points in the windows of $S^{\left[ M^{\prime }\right]
}y^{\prime },$ and the windows occupy at least a $\left( 1-\zeta \right) -$%
fraction of $S^{\left[ M^{\prime }\right] }y^{\prime }.$ So (if $\delta _{1}$
and $\zeta $ are sufficiently small) the set of column blocks that arise
from these $4-$tuples occupy at least a $\left( 1-2\delta \right) -$fraction
of $C^{\prime }.$ Since $C^{\prime }$ is an arbitrary column and $\bar{\mu}%
\times \lambda \left( R^{\prime }\right) >1-\zeta ,$ we obtain (\ref{close P}%
) and (\ref{small alpha}).

To establish (\ref{small break}) we note that $\bar{\mu}\times \lambda
\left( R\backslash R^{\prime }\right) <\zeta $ and if $E$ denotes the set of
points in $R\cap R^{\prime }$ whose ladder block is not used in the
construction of $\bar{R}_{0}$, then $\bar{\mu}\times \lambda \left( E\right)
<\frac{\delta _{1}}{100}.$ Since all the ladder blocks that were used in the
construction of $\bar{R}_{0}$ were unbroken, condition (\ref{small break})
is obtained.

To establish (\ref{dist match}) we argue as in the case of condition (\ref%
{n_1,d_1 regular}). For all points $z\in B_{1}$ whose $G-$coordinate is the
same as that of the initial term in $F,$%
\begin{equation*}
\left\Vert dist_{\bar{S}^{\left[ \left\vert F\right\vert -n_{1}+1\right]
}z}\bigvee_{i\in \left[ n_{1}\right] }\bar{S}_{1}^{-i}(\bar{P}_{1}\vee \bar{%
\alpha}c)\left( z\right) ,dist_{X\times G}\bigvee_{i\in \left[ n_{1}\right]
}S^{-i}\left( P\vee c\right) \right\Vert _{\mathcal{M}}<\frac{\delta _{1}}{%
100}
\end{equation*}%
and for each $h\in G,$%
\begin{equation*}
\left\Vert dist_{\bar{S}^{\left[ \left\vert F\right\vert -n_{1}+1\right]
}zh}\bigvee_{i\in \left[ n_{1}\right] }\bar{S}_{1}^{-i}(\bar{P}_{1}\vee \bar{%
\alpha}c)\left( zh\right) ,dist_{X\times G}\bigvee_{i\in \left[ n_{1}\right]
}S^{-i}\left( P\vee c\right) \right\Vert _{\mathcal{M}}<\frac{\delta _{1}}{%
100}
\end{equation*}%
and this is sufficient to imply (\ref{dist match}).

Finally we verify that condition (\ref{good A}) holds on a sufficiently
large set. Fix a column $C^{\prime }$ of $R^{\prime }.$ Let $D$ denote the
set of real $4-$tuples $\left( t,l,j,i\right) $ associated with $C^{\prime }$
such that:

(a) $\rho ^{\prime }\left[ F_{n}\left( i\right) ,F_{n}\left( \psi \left(
i\right) \right) \right] <\zeta $,

(b) $\rho ^{\prime }\left[ F_{n}\left( \psi \left( i\right) \right)
,\bigvee_{u\in \left[ n\right] }\bar{S}^{-u}\left( \bar{P}\vee \bar{c}%
\right) \left( \tau _{0,t}\psi \left( i\right) \right) \right] <\delta +%
\frac{\delta }{100},$

(c) $\bigvee_{u\in \left[ n\right] }\bar{S}^{-u}\left( \bar{P}\vee \bar{c}%
\vee \boldsymbol{1}_{A_{1}}\right) \left( \tau _{0,t}\psi \left( i\right)
\right) $ and $\bigvee_{u\in \left[ n\right] }\bar{S}^{-u}\left( \bar{P}\vee 
\bar{c}\vee \boldsymbol{1}_{A_{1}}\right) \left( \hat{g}_{j}^{l,t}\left(
i\right) \right) $ are $\zeta -$close.

(d) $\tau _{0,t}\psi \left( i\right) $ is not in one of the miscellaneous
atoms of $\left( Q_{0}\vee \mathcal{\tilde{R}}_{0}\right) $ constructed in
step 1.

As we argued above, conditions (a) and (b) hold for a fraction of real $4-$%
tuples greater than $\left( 1-\zeta \right) $ and $\left( 1-\frac{\delta }{1-%
\frac{\delta _{1}}{100}}\right) >1-2\delta ,$ respectively. By similar
arguments, condition (c) holds for a fraction greater than $\left( 1-2\zeta
\right) $ and (d) holds for a fraction greater than $\left( 1-\frac{%
\varepsilon ^{\prime }/100}{1-\frac{\delta _{1}}{100}}\right) >\left( 1-%
\frac{\varepsilon ^{\prime }}{50}\right) .$ Thus all four conditions hold
for a fraction greater than $\left( 1-\frac{\varepsilon ^{\prime }}{40}%
\right) .$ ($\delta $, $\delta _{1}$and $\zeta $ must have been chosen to
make this last inequality hold).

In the present argument we can ignore the partitions $P$ and $\bar{P}$ in
the above conditions. In particular, when $\left( t,l,j,i\right) \in D$ and
writing $c_{n}\left( i\right) $ for the second component of $F_{n}\left(
i\right) ,$ we have that

(e) $c_{n}\left( i\right) $ and $\bigvee_{u\in \left[ n\right] }\bar{S}^{-u}%
\bar{c}\left( \hat{g}_{j}^{l,t}\left( i\right) \right) $ are uniformly close
to within $2\delta $.

We also note that condition (c) gives

(f) $\bigvee_{u\in \left[ n\right] }\bar{S}^{-u}\left( \boldsymbol{1}%
_{A_{1}}\right) \left( \tau _{0,t}\psi \left( i\right) \right)
=\bigvee_{u\in \left[ n\right] }\bar{S}^{-u}\left( \boldsymbol{1}%
_{A_{1}}\right) \left( \hat{g}_{j}^{l,t}\left( i\right) \right) .$

For each initial triple $\left( t,l,j\right) $ (i.e. the initial triple of a
real $4-$tuple associated with $C^{\prime }),$ and for each $k\in \left[ 
\frac{\left\vert F\right\vert }{n_{1}}\right] ,$ let $\hat{g}_{j,k}^{l,t}$
denote the restriction of $\hat{g}_{j}^{l,t}$ to the interval $\left[ n_{1}%
\right] +kn_{1}.$ For a given triple $\left( t,l,j\right) $ and $k$ we
consider $\mathcal{\tilde{J}}_{k}=\left\{ i\in \mathcal{J}_{k}\mid \left(
t,l,j,i\right) \in D\right\} .$ (We refrain from writing $\mathcal{\tilde{J}}%
_{j,k}^{l,t}$ for $\mathcal{\tilde{J}}_{k}$).

Let $E=\left\{ \left( t,l,j,k\right) \mid \left\vert \mathcal{\tilde{J}}%
_{k}\right\vert >\left( 1-\sqrt{\frac{\varepsilon ^{\prime }}{40}}\right)
\left\vert \mathcal{J}_{k}\right\vert \right\} .$ Then $E$ has density
greater than $\left( 1-\sqrt{\frac{\varepsilon ^{\prime }}{40}}\right) $ in
the set of all $\left( t,l,j,k\right) .$

Fix $\left( t,l,j,k\right) \in E.$ The range of $\hat{g}_{j,k}^{l,t}$ is an
orbit of $\bar{S}_{1}$ whose initial point we denote by $y.$ Let $C^{\prime
}\left( y\right) $ denote the level of $C^{\prime }$ containing $y.$ Then $%
C^{\prime }\left( y\right) $ is the base of an $\bar{S}_{1}-n_{1}-$column
and is a subset of $\Lambda _{n_{1}}\left( \bar{S}_{1}\right) $. We will
show that (\ref{good A}) holds for all points in $C^{\prime }\left( y\right) 
$.

First we consider the point $y$ itself. If it were the case that $\mathcal{J}%
_{k}=\mathcal{\tilde{J}}_{k}$, then for all $i\in \mathcal{J}_{k}$ condition
(e) would hold. In addition, if for all $q\in Q$ we let $\mathcal{J}%
_{k,q}=\left\{ i\in \mathcal{J}_{k}\mid F_{n}\psi \left( i\right) \in
q\right\} ,$ we would have

\begin{equation*}
\left\Vert dist_{\mathcal{J}_{k,q}}\bigvee_{u\in \left[ n\right] }\bar{S}%
^{-u}\left( \boldsymbol{1}_{A_{1}}\right) \left( \hat{g}_{j}^{l,t}\right)
,dist_{q\cap \mathcal{V}_{0}}\bigvee_{u\in \left[ n\right] }\bar{S}%
^{-u}\left( \boldsymbol{1}_{A_{1}}\right) \right\Vert _{\mathcal{M}}<\frac{%
\varepsilon ^{\prime }}{50}
\end{equation*}%
(since the miscellaneous atoms occupy less than an $\frac{\varepsilon
^{\prime }}{100}-$fraction of $\mathcal{V}_{0}$). Therefore, we would have%
\begin{equation*}
\left\Vert dist_{\mathcal{J}_{k}}\bigvee_{u\in \left[ n\right] }\bar{S}%
^{-u}\left( \bar{c}\vee \boldsymbol{1}_{A_{1}}\right) \left( \hat{g}%
_{j}^{l,t}\right) ,dist_{\mathcal{V}_{0}}\bigvee_{u\in \left[ n\right] }\bar{%
S}^{-u}\left( \bar{c}\vee \boldsymbol{1}_{A_{1}}\right) \right\Vert _{%
\mathcal{M}}<\frac{\varepsilon ^{\prime }}{50}+\delta .
\end{equation*}%
(since the $c-n-$names can differ by $\delta $). This would give%
\begin{equation*}
\left\Vert dist_{\mathcal{J}_{k}}\bigvee_{u\in \left[ n\right] }\bar{S}%
^{-u}\left( \bar{c}\vee \boldsymbol{1}_{A_{1}}\right) \left( \hat{g}%
_{j}^{l,t}\right) ,dist_{\bar{X}\times G}\bigvee_{u\in \left[ n\right] }\bar{%
S}^{-u}\left( \bar{c}\vee \boldsymbol{1}_{A_{1}}\right) \right\Vert _{%
\mathcal{M}}<\frac{\varepsilon ^{\prime }}{50}+\delta +\zeta .
\end{equation*}%
But since we only have that $\left\vert \mathcal{\tilde{J}}_{k}\right\vert
>\left( 1-\sqrt{\frac{\varepsilon ^{\prime }}{40}}\right) \left\vert 
\mathcal{J}_{k}\right\vert ,$ we have instead%
\begin{equation*}
\left\Vert dist_{\mathcal{J}_{k}}\bigvee_{u\in \left[ n\right] }\bar{S}%
^{-u}\left( \bar{c}\vee \boldsymbol{1}_{A_{1}}\right) \left( \hat{g}%
_{j}^{l,t}\right) ,dist_{\bar{X}\times G}\bigvee_{u\in \left[ n\right] }\bar{%
S}^{-u}\left( \bar{c}\vee \boldsymbol{1}_{A_{1}}\right) \right\Vert _{%
\mathcal{M}}<\frac{\varepsilon ^{\prime }}{50}+\delta +\zeta +\sqrt{\frac{%
\varepsilon ^{\prime }}{40}}<\sqrt{\frac{\varepsilon ^{\prime }}{30}}
\end{equation*}%
(if $\varepsilon ^{\prime }$, $\delta $, and $\zeta $ are sufficiently
small).

For each $i\in \mathcal{J}_{k}$ we write $n\left( i\right) $ for the $n-$%
block begining at $i$ and $n\left( \mathcal{J}_{k}\right) $ for $%
\tbigcup_{i\in \mathcal{J}_{k}}n\left( i\right) .$ Then we have%
\begin{equation*}
\left\Vert dist_{n\left( \mathcal{J}_{k}\right) }\left( \bar{c}\vee 
\boldsymbol{1}_{A_{1}}\right) \left( \hat{g}_{j}^{l,t}\right) ,dist_{\bar{X}%
\times G}\left( \bar{c}\vee \boldsymbol{1}_{A_{1}}\right) \right\Vert _{%
\mathcal{M}}<\sqrt{\frac{\varepsilon ^{\prime }}{30}}.
\end{equation*}%
Since $\left\vert n\left( \mathcal{J}_{k}\right) \right\vert >\left( 1-\zeta
\right) n_{1}$ we get 
\begin{equation*}
\left\Vert dist_{\left[ n_{1}\right] +kn_{1}}\left( \bar{c}\vee \boldsymbol{1%
}_{A_{1}}\right) \left( \hat{g}_{j}^{l,t}\right) ,dist_{\bar{X}\times
G}\left( \bar{c}\vee \boldsymbol{1}_{A_{1}}\right) \right\Vert _{\mathcal{M}%
}<\sqrt{\frac{\varepsilon ^{\prime }}{20}}.
\end{equation*}%
In other words,%
\begin{equation*}
\left\Vert dist_{\bar{S}_{1}^{\left[ n_{1}\right] }y}\left( \bar{c}\vee 
\boldsymbol{1}_{A_{1}}\right) ,dist_{\bar{X}\times G}\left( \bar{c}\vee 
\boldsymbol{1}_{A_{1}}\right) \right\Vert _{\mathcal{M}}<\sqrt{\frac{%
\varepsilon ^{\prime }}{20}}.
\end{equation*}%
Since $A_{1}\in \mathcal{\bar{A}}$, $\bar{c}$ is independent of $\boldsymbol{%
1}_{A_{1}},$ so for all $h\in G,$ $dist_{\bar{X}\times G}\left( \bar{c}h\vee 
\boldsymbol{1}_{A_{1}}\right) =dist_{\bar{X}\times G}\left( \bar{c}\vee 
\boldsymbol{1}_{A_{1}}\right) $ and $dist_{\bar{S}_{1}^{\left[ n_{1}\right]
}yh}\left( c\vee \boldsymbol{1}_{A_{1}}\right) =dist_{\bar{S}_{1}^{\left[
n_{1}\right] }y}\left( ch\vee \boldsymbol{1}_{A_{1}}\right) ,$ so 
\begin{equation*}
\left\Vert dist_{\bar{S}_{1}^{\left[ n_{1}\right] }yh}\left( \bar{c}\vee 
\boldsymbol{1}_{A_{1}}\right) ,dist_{\bar{X}\times G}\left( \bar{c}\vee 
\boldsymbol{1}_{A_{1}}\right) \right\Vert _{\mathcal{M}}<\sqrt{\frac{%
\varepsilon ^{\prime }}{20}}.
\end{equation*}%
If $\tilde{y}\in C^{\prime }\left( y\right) $ and $\bar{c}\left( \tilde{y}%
\right) =\bar{c}\left( y\right) ,$ then for all $u\in \left[ n_{1}\right] ,$ 
$\rho \left( \bar{c}\left( \bar{S}_{1}^{u}\tilde{y}\right) ,\bar{c}\left( 
\bar{S}_{1}^{u}y\right) \right) <\zeta $ so we get 
\begin{equation*}
\left\Vert dist_{\bar{S}_{1}^{\left[ n_{1}\right] }\tilde{y}}\left( \bar{c}%
\vee \boldsymbol{1}_{A_{1}}\right) ,dist_{\bar{X}\times G}\left( \bar{c}\vee 
\boldsymbol{1}_{A_{1}}\right) \right\Vert _{\mathcal{M}}<\sqrt{\frac{%
\varepsilon ^{\prime }}{20}}+\zeta
\end{equation*}%
and so as before, for all $h\in G$%
\begin{equation*}
\left\Vert dist_{\bar{S}_{1}^{\left[ n_{1}\right] }\tilde{y}h}\left( \bar{c}%
\vee \boldsymbol{1}_{A_{1}}\right) -dist_{\bar{X}\times G}\left( \bar{c}\vee 
\boldsymbol{1}_{A_{1}}\right) \right\Vert _{\mathcal{M}}<\sqrt{\frac{%
\varepsilon ^{\prime }}{20}}+\zeta .
\end{equation*}%
For simplicity, let's suppose that here we got $\varepsilon ^{\prime }$ as
an upper estimate, as we could have done. Then for all $z\in C^{\prime
}\left( y\right) $ we have 
\begin{equation*}
\frac{1}{n_{1}}\sum \boldsymbol{1}_{A_{1}}\left( \bar{S}_{1}^{u}z\right) >%
\bar{\mu}\left( A_{1}\right) -\varepsilon ^{\prime }
\end{equation*}%
and conditioning on $\left\{ u\mid \bar{S}_{1}^{u}z\in A_{1}\right\} ,$ we
have 
\begin{equation*}
\left\Vert dist_{\left\{ u\mid \bar{S}_{1}^{u}z\in A_{1}\right\} }\bar{c}%
\left( \left( \bar{S}_{1}^{u}z\right) \right) -\lambda \right\Vert _{%
\mathcal{M}}<\varepsilon ^{\prime }
\end{equation*}%
so that (by the choice of $\varepsilon ^{\prime }$) 
\begin{equation*}
\left\Vert dist_{\left\{ u\mid \bar{S}_{1}^{u}z\in A_{1}\right\} }\bar{c}%
\left( \left( \bar{S}_{1}^{u}z\right) \right) -\lambda \left( A_{2}\right)
\right\Vert _{\mathcal{M}}<\frac{\varepsilon }{100}
\end{equation*}%
and so 
\begin{equation*}
\frac{1}{n_{1}}\sum \boldsymbol{1}_{A}\left( \bar{S}_{1}^{u}z\right) >\bar{%
\mu}\times \lambda \left( A_{1}\right) -\varepsilon .
\end{equation*}

Thus for all $\left( t,l,j,k\right) $ as above, and all $z$ in the base of
the corresponding $\bar{S}_{1}-n_{1}-$column block, condition (\ref{good A})
holds. But the union of these column blocks covers at least a $\left( 1-%
\sqrt{\frac{\varepsilon ^{\prime }}{40}}\right) -$fraction of the portion of 
$\bar{R}_{0}\cap C^{\prime }$ and hence at least a $\left( 1-4\delta \right)
\left( 1-\sqrt{\frac{\varepsilon ^{\prime }}{40}}\right) -$fraction of $\bar{%
R}_{1}\cap C^{\prime }.$ Therefore, on at least a $\left( 1-\varepsilon
\right) -$fraction of $\Lambda _{n_{1}}\left( \bar{S}_{1}\right) \cap
C^{\prime }$ we get condition (\ref{good A}). Since the same argument
applies to every other column of $R^{\prime }$, the argument is complete.
\end{proof}

The above lemma will now be extended to the case of countable partitions.
The statement is identical, with the understanding that the partitions $P$
and $\bar{P}$ are countably infinite.

\begin{lemma}
\label{ctble dil}(Distribution Improvement Lemma for countable partitions)

\begin{proof}
Fix $\varepsilon >0$ and $A_{2}\subset G$ open. Let $n$ and $\delta $ be
given by lemma \ref{dil} with respect to $\varepsilon $ and $A_{2}$. Suppose
that $\left( S,T,\sigma ,X\right) $ and $\left( \bar{S}_{0},\bar{T}_{0},\bar{%
\sigma}_{0},\bar{X}\right) $ are\ ergodic $G-$extensions on $\left( X\times
G\right) $ and $\left( \bar{X}\times G\right) $ respectively with partitions 
$P$ and $\bar{P}$ are as in the statement of this lemma. Fix $n_{1},\delta
_{1}$ and $A_{1}\subset \bar{X}.$ We assume that the elements of $P$ are
indexed by $\mathbb{N},$ and we let $P_{N}$ denote the partition formed by
replacing the set of elements of $P$ indexed by integers greater than $N$ by
their union. Choose $N$ so large that 
\begin{equation}
\left\Vert dist_{X\times G}\bigvee_{i\in \left[ n_{1}\right] }S^{-i}\left(
P_{N}\vee c\right) ,dist_{X\times G}\bigvee_{i\in \left[ n_{1}\right]
}S^{-i}\left( P\vee c\right) \right\Vert _{\mathcal{M}}<\frac{\delta _{1}}{2}
\label{n1dist for PN}
\end{equation}%
and \ 
\begin{equation}
\left\Vert dist_{X\times G}\bigvee_{i\in \left[ n\right] }S^{-i}\left(
P_{N}\vee c\right) ,dist_{Dom\left( \bar{S}^{n}\right) }\bigvee_{i\in \left[
n\right] }\bar{S}^{-i}\left( \bar{P}\vee c\right) \right\Vert _{\mathcal{M}%
}<\delta .
\end{equation}%
Apply lemma \ref{dil} to the systems $\left( S,P_{N}\right) ,\left( \bar{S},%
\bar{P}\right) ,\left( \bar{S}_{0}\right) ,$ and the set $A_{1}\times A_{2}$
but using the parameter $\frac{\delta _{1}}{2}$ instead of $\delta _{1}.$
This gives us a new partial speedup $\bar{S}_{1}$ of $\bar{S}_{0},$ a
partition $\bar{P}_{1}$ and a function $\bar{\alpha}$ satisfying the
conclusions of the lemma \ref{dil} (with respect to $\left( S,P_{N}\right) $%
). But then condition \ref{n1dist for PN} gives the conclusion \ref{dist
match}. All the other conclusions don't refer to $P,$ so they are met as
well.
\end{proof}
\end{lemma}

We note that the order of quantifiers in the statement of lemma \ref{dil},
namely that $\delta $ and $n$ depend only on $\varepsilon $ and $A_{2}$, was
used in an essential way in this argument.

\section{Factor theorem}

Our goal here is to obtain the following:

\begin{theorem}
\label{factor b}For all $\varepsilon >0$ there exists $\delta >0$ and $n\in 
\mathbb{N}$ such that if $\left( S,T,\sigma ,X\right) $ and $\left( \bar{S}%
_{0},\bar{T}_{0},\bar{\sigma}_{0},\bar{X}\right) $ are ergodic $G-$%
extensions on $X\times G$ and $\bar{X}\times G,$ respectively and if $P$ is
a generator for $\left( T,X\right) $ and $\bar{P}$ is a partition of $\bar{X}%
,$ such that 
\begin{equation}
\left\Vert dist_{X\times G}\bigvee_{i\in \left[ n\right] }S^{-i}\left( P\vee
c\right) ,dist_{\bar{X}\times G}\bigvee_{i\in \left[ n\right] }\bar{S}%
_{0}^{-i}\left( \bar{P}\vee c\right) \right\Vert _{\mathcal{M}}<\delta ,
\end{equation}%
then there exists an ergodic $G-$speedup $\left( \hat{S},\hat{T},\hat{\sigma}%
,\bar{X}\right) $ of $\left( \bar{S}_{0},\bar{T}_{0},\bar{\sigma}_{0},\bar{X}%
\right) $ and a partition $\hat{P}$ of $\bar{X}$ and a measurable function $%
\hat{\alpha}:\bar{X}\rightarrow G$ such that $\left\vert \bar{P}-\hat{P}%
\right\vert <\varepsilon ,$ 
\begin{equation*}
\int_{\bar{X}}\rho \left( \hat{\alpha}\left( \bar{x}\right) ,id_{G}\right) d%
\bar{\mu}<\varepsilon ,
\end{equation*}
\begin{equation*}
\bar{\mu}\times \lambda \left\{ \left( \bar{x},g\right) \mid \hat{S}\left( 
\bar{x},g\right) \neq \bar{S}_{0}\left( \bar{x},g\right) \right\}
<\varepsilon ,
\end{equation*}
and for all $n\in \mathbb{N},$%
\begin{equation*}
\left\Vert dist_{X\times G}\bigvee_{i\in \left[ n\right] }S^{-i}\left( P\vee
c\right) ,dist_{\bar{X}\times G}\bigvee_{i\in \left[ n\right] }\hat{S}%
^{-i}\left( \hat{P}\vee \hat{\alpha}c\right) \right\Vert _{\mathcal{M}}=0.
\end{equation*}%
In particular, $\left( \hat{S},\hat{T},\hat{\sigma},\bar{X}\right) $ has $%
\left( S,T,\sigma ,X\right) $ as a $G-$factor.
\end{theorem}

This theorem will follow from the next lemma, which will be proved by
repeated application of lemma \ref{ctble dil}.

\begin{lemma}
\label{factor c}For all $\varepsilon >0$ there exists $\delta >0$ and $n\in 
\mathbb{N}$ so that if $\left( S,T,\sigma ,X\right) $ and $\left( \bar{S}%
_{0},\bar{T}_{0},\bar{\sigma}_{0},\bar{X}\right) $ are ergodic $G-$%
extensions on $X\times G$ and $\bar{X}\times G,$ respectively, and $\left( 
\bar{S},\bar{T},\bar{\sigma},\bar{X}\right) $ is a partial $G-$speedup of $%
\bar{S}_{0},$ and $P$ and $\bar{P}$ are partitions of $X$ and $\bar{X},$
respectively, such that $\left( \bar{S},\bar{P}\right) $ is $\left( n,\delta
\right) -$regular, and 
\begin{equation*}
\left\Vert dist_{X\times G}\bigvee_{i\in \left[ n\right] }S^{-i}\left( P\vee
c\right) ,dist_{Dom\left( \bar{S}^{n}\right) }\bigvee_{i\in \left[ n\right] }%
\bar{S}^{-i}\left( \bar{P}\vee c\right) \right\Vert _{\mathcal{M}}<\delta ,
\end{equation*}%
then there is an ergodic $G-$speedup $\left( \hat{S},\hat{T},\hat{\sigma},%
\bar{X}\right) $ of $\left( \bar{S}_{0},\bar{T}_{0},\bar{\sigma}_{0},\bar{X}%
\right) $ and a partition $\hat{P}$ of $\bar{X}$ and a measurable function $%
\hat{\alpha}:\bar{X}\rightarrow G$ such that 
\begin{equation}
\left\vert \bar{P}-\hat{P}\right\vert <\varepsilon ,  \label{close P fac}
\end{equation}%
\begin{equation}
\int_{\bar{X}}\rho \left( \hat{\alpha}\left( \bar{x}\right) ,id_{G}\right) d%
\bar{\mu}<\varepsilon ,  \label{small alpha fac}
\end{equation}%
\begin{equation}
\bar{\mu}\times \lambda \left\{ \left( \bar{x},g\right) \mid \hat{S}\left( 
\bar{x},g\right) \neq \bar{S}_{0}\left( \bar{x},g\right) \right\}
<\varepsilon ,  \label{small break fac}
\end{equation}%
and for all $n\in \mathbb{N},$%
\begin{equation}
\left\Vert dist_{X\times G}\bigvee_{i\in \left[ n\right] }S^{-i}\left( P\vee
c\right) ,dist_{\bar{X}\times G}\bigvee_{i\in \left[ n\right] }\hat{S}%
^{-i}\left( \hat{P}\vee \hat{\alpha}c\right) \right\Vert _{\mathcal{M}}=0.
\label{equal dist fac}
\end{equation}
\end{lemma}

\begin{proof}
Fix $\varepsilon >0.$ Choose $\varepsilon _{k}>0$ so that $%
\sum_{k=0}^{\infty }\varepsilon _{k}<\frac{\varepsilon }{2}.$ Fix a sequence
of measurable rectangles $\left\{ A^{\left( k\right) }=A_{1}^{\left(
k\right) }\times A_{2}^{\left( k\right) }\right\} _{k=0}^{\infty }$ in $\bar{%
X}\times G,$ where each $A_{2}^{\left( k\right) }$ is open, each rectangle
appears infinitely often in the sequence, and the sequence is dense in the
measure algebra of $\left( \bar{X}\times G,\bar{\mu}\times \lambda \right) .$
For each $k,$ choose $\delta _{k}<\frac{\varepsilon _{k}}{2}$ and $n_{k}$ by
applying lemma \ref{ctble dil} with respect to $\varepsilon _{k}$ and $%
A_{2}^{\left( k\right) }.$ We will see that $\delta _{0}$ and $n_{0}$ serve
as the $\delta $ and $n$ in the conclusion of this theorem.

Suppose now that $\left( S,T,\sigma ,X\right) $ and $\left( \bar{S}_{0},\bar{%
T}_{0},\bar{\sigma}_{0},\bar{X}\right) $ are given as in the statement of
the lemma$.$ Let $\bar{S}_{1},\bar{P}_{1}$ and $\bar{\alpha}_{1}$ be the
partial speedup of $\bar{S}_{0},$ the partition and the function given by
lemma \ref{ctble dil}. Here we use $A^{\left( 0\right) }=A_{1}^{\left(
0\right) }\times A_{2}^{\left( 0\right) }$ as the rectangle in the statement
of lemma \ref{ctble dil}. Roughly speaking, lemma \ref{ctble dil} allows us
to make an $\varepsilon _{0}-$small modification of $\left( \bar{S},\bar{P}%
\right) $ to obtain $\left( \bar{S}_{1},\bar{P}_{1}\right) .$ The
conclusions of lemma \ref{ctble dil} and the choice of $\delta _{1}$ and $%
n_{1}$ allow us to apply lemma \ref{ctble dil} again to the new $G$
extension $\bar{S}_{0}^{\bar{\alpha}_{1}}$ and its partial speedup $\bar{S}%
_{1}^{\bar{\alpha}_{1}}\,$to obtain $\bar{S}_{2},\bar{P}_{2}$ and $\bar{%
\alpha}_{2}$ which meet the conclusions of lemma \ref{ctble dil} with
respect to the rectangle $A^{\left( 1\right) }=A_{1}^{\left( 1\right)
}\times A_{2}^{\left( 1\right) }.$ That is, we make an $\varepsilon _{1}-$%
small modification of $\left( \bar{S}_{1}^{\bar{\alpha}_{1}},\bar{P}%
_{1}\right) $ and obtain a new partial speedup $\bar{S}_{2}$ of $\bar{S}_{0}$%
, a partition $\bar{P}_{2}$ and a function $\bar{\alpha}_{2}.$ At this
point, we let $\bar{\beta}_{2}=\bar{\alpha}_{2}\bar{\alpha}_{1}$ so that the 
$G$ extension $\bar{S}_{0}^{\bar{\beta}_{2}}$ and its partial speedup $\bar{S%
}_{2}^{\bar{\beta}_{2}}$ meet the conditions to which lemma \ref{ctble dil}
can be applied once again.

Continuing in this way we obtain a sequence of speedups $\bar{S}_{k},$
partitions $\bar{P}_{k}$ and functions $\bar{\alpha}_{k}$ so that, writing $%
\bar{\beta}_{k}=\tprod\limits_{j=0}^{k-1}\bar{\alpha}_{n-j},$ we have for
each $k,$%
\begin{equation}
\left\vert \bar{P}_{k+1}-\bar{P}_{k}\right\vert <\varepsilon _{k},
\label{close pk fac}
\end{equation}%
\begin{equation}
\int_{\bar{X}}\rho \left( \bar{\alpha}_{k}\left( \bar{x}\right)
,id_{G}\right) d\bar{\mu}<\varepsilon _{k},  \label{small alphak fac}
\end{equation}

if $D_{k}$ denotes the set of points in the speedup tower of $\bar{S}_{k}$
such whose ladder block is broken by $\bar{S}_{k+1},$ then%
\begin{equation}
\bar{\mu}\times \lambda \left( D_{k}\right) <\varepsilon _{k},
\label{small breakk fac}
\end{equation}

\begin{equation}
\left\Vert dist_{X\times G}\bigvee_{i\in \left[ n_{k}\right] }S^{-i}\left(
P\vee c\right) ,dist_{Dom\left( \bar{S}_{k}^{n_{k}}\right) }\bigvee_{i\in %
\left[ n_{k}\right] }\bar{S}_{k}^{-i}\left( \bar{P}_{k}\vee \bar{\beta}%
_{k}c\right) \right\Vert _{\mathcal{M}}<\delta _{k},  \label{close kdist fac}
\end{equation}

and the set of $y\in \Lambda _{n_{k}}\left( \bar{S}_{k}\right) $ such that 
\begin{equation}
\frac{1}{n_{k}}\sum_{i\in \left[ n_{k}\right] }\chi _{A^{\left( k\right)
}}\left( \bar{S}_{k}^{i}\left( y\right) \right) >\left( \bar{\mu}\times
\lambda \right) \left( A^{\left( k\right) }\right) -\varepsilon _{k}
\label{good Ak fac}
\end{equation}%
has measure greater than $\left( 1-\varepsilon _{k}\right) \left( \bar{\mu}%
\times \lambda \right) \left( \Lambda _{n_{k}}\left( \bar{S}_{k}\right)
\right) .$

(Recall that $\Lambda _{n_{k}}\left( \bar{S}_{k}\right) $ denotes the $%
n_{k}- $ladder in the speedup tower of $\bar{S}_{k}$).

Conditions \ref{close pk fac} and \ref{small breakk fac} and the fact that,
for each $k,$ $\bar{\mu}\times \lambda \left( Dom\left( \bar{S}_{k}\right)
\right) >1-\varepsilon _{k}$ imply that there is a partition $\hat{P}$ and a
function $\hat{\alpha}$ such that $\lim_{k\rightarrow \infty }\left\vert 
\bar{P}_{k}-\hat{P}\right\vert =0$ and $\lim_{k\rightarrow \infty }\bar{\beta%
}_{k}=\hat{\alpha}$ a.e., where $\hat{P}$ and $\hat{\alpha}$ satisfy
conditions \ref{close P fac} and \ref{small alpha fac}.

Condition \ref{small breakk fac} implies (again using $\bar{\mu}\times
\lambda \left( Dom\left( \bar{S}_{k}\right) \right) >1-\varepsilon _{k}$)
that the partial transformations $\bar{S}_{k}$ converge almost everywhere to
a transformation $\hat{S}$ that is a $G-$speedup of $\bar{S}_{0},$ and $\hat{%
S}$ satisfies \ref{small break fac}.

To establish condition \ref{equal dist fac} we fix $n^{\prime },$ and $%
\delta ^{\prime }$ and choose $k$ so that $n_{k}>n^{\prime }.$ We know that 
\begin{equation*}
\left\Vert dist_{X\times G}\bigvee_{i\in \left[ n_{k}\right] }S^{-i}\left(
P\vee c\right) ,dist_{Dom\left( \bar{S}_{k}^{n_{k}}\right) }\bigvee_{i\in %
\left[ n_{k}\right] }\bar{S}_{k}^{-i}\left( \bar{P}_{k}\vee \bar{\beta}%
_{k}c\right) \right\Vert _{\mathcal{M}}<\delta _{k}
\end{equation*}%
and so the same is true for the $n^{\prime }$ distribution:%
\begin{equation*}
\left\Vert dist_{X\times G}\bigvee_{i\in \left[ n^{\prime }\right]
}S^{-i}\left( P\vee c\right) ,dist_{Dom\left( \bar{S}_{k}^{n}\right)
}\bigvee_{i\in \left[ n^{\prime }\right] }\bar{S}_{k}^{-i}\left( \bar{P}%
_{k}\vee \bar{\beta}_{k}c\right) \right\Vert _{\mathcal{M}}<\delta _{k}
\end{equation*}%
Moreover, the set of $\bar{S}_{k}-$ladder blocks that are broken by $\hat{S}$
has measure less than $\sum_{i=k}^{\infty }\varepsilon _{i},$ and $%
\left\vert \bar{P}_{k}-\hat{P}\right\vert <\sum_{i=k}^{\infty }\varepsilon
_{i},$ and $\int_{\bar{X}}\rho \left( \bar{\beta}_{k}\left( \bar{x}\right) ,%
\hat{\alpha}\left( \bar{x}\right) \right) d\bar{\mu}<\sum_{i=k}^{\infty
}\varepsilon _{i}.$ Since the measure of the speedup tower for $\bar{S}_{k}$
is greater than $1-\frac{\varepsilon _{k}}{2},$ we see that if $k$ is
sufficiently large, (so that the set of points whose $\bar{S}_{k}-n-$orbits
are not wholly contained in a ladder block for $\bar{S}_{k}$ is small), we
get 
\begin{equation*}
\left\Vert dist_{X\times G}\bigvee_{i\in \left[ n^{\prime }\right]
}S^{-i}\left( P\vee c\right) ,dist_{\bar{X}\times G}\bigvee_{i\in \left[
n^{\prime }\right] }\hat{S}^{-i}\left( \hat{P}\vee \hat{\alpha}c\right)
\right\Vert _{\mathcal{M}}<\delta ^{\prime }
\end{equation*}%
Since this is true for all $n^{\prime }$ and $\delta ^{\prime },$ we have
condition \ref{equal dist fac}.

Finally, we show that $\hat{S}$ is ergodic. Fix rectangles $A^{\left(
i\right) }$ and $A^{\left( j\right) }$ where 
\begin{equation*}
\left( \bar{\mu}\times \lambda \right) \left( A^{\left( i\right) }\right)
<\left( \bar{\mu}\times \lambda \right) \left( A^{\left( j\right) }\right) .
\end{equation*}%
Condition \ref{good Ak fac}, and the fact that the measure of the speedup
tower for $\bar{S}_{k}$ is greater than $1-\frac{\varepsilon _{k}}{2}$
implies that for all $\varepsilon ^{\prime }$ there exists $k$ such that the
set of $\bar{S}_{k}-$ladder blocks on which $A^{\left( i\right) }$ has
density within $\varepsilon ^{\prime }$ of the measure of $A^{\left(
i\right) }$ exceeds $1-\varepsilon ^{\prime }.$ For $l>k,$ most $\bar{S}%
_{l}- $ ladder blocks are mostly covered by these $\bar{S}_{k}-$ladder
blocks, so we get the stronger fact that for all $\varepsilon ^{\prime }$
and for all sufficiently large $k,$ the set of $\bar{S}_{k}-$ladder blocks
on which $A^{\left( i\right) }$ has density within $\varepsilon ^{\prime }$
of the measure of $A^{\left( i\right) }$ exceeds $1-\varepsilon ^{\prime }.$
Applying this to both $A^{\left( i\right) }$ and $A^{\left( j\right) },$ we
can choose $k$ so that the above condition holds for both rectangles, and in
addition, the set of $\bar{S}_{k}-$ ladder blocks that are broken by $\hat{S}
$ has measure less than $\varepsilon ^{\prime }.$ If $\varepsilon ^{\prime }$
is small enough, we conclude that there is a transformation $S^{\prime }$ in
the full group of $\hat{S}$ so that $\left( \bar{\mu}\times \lambda \right)
\left( S^{\prime }\left( A^{\left( i\right) }\right) \cap A^{\left( j\right)
}\right) >\left( 1-\varepsilon ^{\prime }\right) \left( \bar{\mu}\times
\lambda \right) \left( A^{\left( i\right) }\right) .$ We conclude from lemma %
\ref{ergodicity} that $\hat{S}$ is ergodic.
\end{proof}

We now give the proof of theorem \ref{factor b} using lemma \ref{factor c}

\begin{proof}
(of theorem \ref{factor b}$)$ Fix $\varepsilon >0.$ Choose $\delta $ and $n$
by lemma \ref{factor c} with respect to $\frac{\varepsilon }{2}.$ Suppose
that $\left( S,T,\sigma ,X\right) $ and $\left( \bar{S}_{0},\bar{T}_{0},\bar{%
\sigma}_{0},\bar{X}\right) $ are ergodic $G-$extensions and $P$ and $\bar{P}$
are partitions satisfying the hypotheses of theorem \ref{factor b}, but
where the distribution match is to within $\frac{\delta }{3}$. Lemma \ref%
{model name 1} gives the following: For all $\zeta >0$ there exists $L\left(
\zeta \right) \in \mathbb{N}$ so that \bigskip \newline
1. for all $L\geq L\left( \zeta \right) ,$ $\left( 1-\zeta \right) -$most
points $\bar{x}\in \bar{X}$ have the property that for all $g\in G,$ 
\begin{equation*}
\left\Vert dist_{\bar{S}_{0}^{\left[ L-n+1\right] }\left( \bar{x},g\right)
}\bigvee_{i\in \left[ n\right] }\bar{S}_{0}^{-i}\left( \bar{P}\vee c\right)
,dist_{\bar{X}\times G}\bigvee_{i\in \left[ n\right] }\bar{S}_{0}^{-i}\left( 
\bar{P}\vee c\right) \ \right\Vert _{\mathcal{M}}<\zeta
\end{equation*}%
and \bigskip \newline
2. the interval $\left[ L-1\right] $ can be $\left( 1-\zeta \right) -$%
disjointly covered by a set of intervals of length $n,$ so that (again for
all $g\in G$) if $J\left( \bar{x}\right) $ is the set of initial integers of
these intervals, 
\begin{equation*}
\left\Vert dist_{\bar{S}_{0}^{\left[ J\left( \bar{x}\right) \right]
}}\bigvee_{i\in \left[ n\right] }\bar{S}_{0}^{-i}\left( \bar{P}\vee c\right)
,dist_{\bar{X}\times G}\bigvee_{i\in \left[ n\right] }\bar{S}_{0}^{-i}\left( 
\bar{P}\vee c\right) \ \right\Vert _{\mathcal{M}}<\zeta
\end{equation*}

In addition, these $n-$blocks are organized into groups of consecutive $n-$%
blocks where these groups can be chosen to be as long as we please.
Consequently, for each such $\bar{x}$ we can speed up the $\bar{T}-L-$orbit
of $\bar{x}$ (and correspondingly speed up the $\bar{S}_{0}-L-$orbit of $%
\left( \bar{x},g\right) ,$ for each $g\in G$) by skipping over any points
that are not in the orbit blocks chosen by these intervals. If the lengths
of the consecutive groups of $n$ blocks are sufficiently large compared to $%
n $, then the distribution of $\left( \bar{S}_{0},P\vee c\right) -n-$names
on such a $G-$speedup orbit segment will be $2\zeta $ close to the
distribution of $\left( \bar{S}_{0},P\vee c\right) -n-$names on $\bar{X}%
\times G.$

Now choose a Rokhlin tower for $\bar{S}_{0}$, measurable with respect to $%
\bar{X}$ and of height $L$ so that all points in its base are of the above
type. For each point $\bar{x}$ of the base, implement the speedup described
above, and remove just enough levels from the top of the orbit segment above 
$\bar{x}$ so that the tower which remains is of constant height $L^{\prime }$%
, where $L^{\prime }$ is a multiple of $n.$ If $\zeta $ was chosen
sufficiently small, this gives a $G-$speedup $\bar{S}$ and a $\left( \delta
,n\right) -$ regular speedup tower that satisfy the hypotheses of lemma \ref%
{factor c}. In addition, we may arrange that 
\begin{equation*}
\bar{\mu}\times \lambda \left\{ \left( \bar{x},g\right) \mid \bar{S}\left( 
\bar{x},g\right) \neq \bar{S}_{0}\left( \bar{x},g\right) \right\} <\frac{%
\varepsilon }{2}
\end{equation*}%
Consequently, lemma \ref{factor c} gives us an ergodic speedup $\left( \hat{S%
},\hat{T},\hat{\sigma},\bar{X}\right) $ satisfying the conclusions of
theorem \ref{factor b}
\end{proof}

\section{Isomorphism theorem}

We now wish to prove our main theorem:

\begin{theorem}
\label{isomorphism a}Let $\left( S,T,\sigma ,X\right) $ and $\left( \bar{S}%
_{0},\bar{T}_{0},\bar{\sigma}_{0},\bar{X}\right) $ be ergodic $G-$extensions
on $X\times G$ and $\bar{X}\times G,$ respectively. Then for all $%
\varepsilon >0$ there exists an ergodic $G-$speedup $\left( \hat{S},\hat{T},%
\hat{\sigma},\bar{X}\right) $ of $\left( \bar{S}_{0},\bar{T}_{0},\bar{\sigma}%
_{0},\bar{X}\right) $ such that 
\begin{equation*}
\bar{\mu}\times \lambda \left\{ \left( \bar{x},g\right) \mid \hat{S}\left( 
\bar{x},g\right) \neq \bar{S}_{0}\left( \bar{x},g\right) \right\}
<\varepsilon ,
\end{equation*}%
and $\left( \hat{S},\hat{T},\hat{\sigma},\bar{X}\right) $ and $\left(
S,T,\sigma ,X\right) $ are $G-$isomorphic.
\end{theorem}

We will first prove a version of this theorem analogous to theorem \ref%
{factor b}, and then use it to obtain theorem \ref{isomorphism a}.

\begin{theorem}
\label{isomorphism b}For all $\varepsilon >0$ there exists $\delta >0$ and $%
n\in \mathbb{N}$ such that if $\left( S,T,\sigma ,X\right) $ and $\left( 
\bar{S}_{0},\bar{T}_{0},\bar{\sigma}_{0},\bar{X}\right) $ are ergodic $G-$%
extensions on $X\times G$ and $\bar{X}\times G,$ respectively and if $P$ is
a generator for $\left( T,X\right) $ and $\bar{P}$ is a partition of $\bar{X}%
,$ such that 
\begin{equation}
\left\Vert dist_{X\times G}\bigvee_{i\in \left[ n\right] }S^{-i}\left( P\vee
c\right) ,dist_{\bar{X}\times G}\bigvee_{i\in \left[ n\right] }\bar{S}%
_{0}^{-i}\left( \bar{P}\vee c\right) \right\Vert _{\mathcal{M}}<\delta ,
\label{dist hypoth isom b}
\end{equation}%
then there exists an ergodic $G-$speedup $\left( \hat{S},\hat{T},\hat{\sigma}%
,\bar{X}\right) $ of $\left( \bar{S}_{0},\bar{T}_{0},\bar{\sigma}_{0},\bar{X}%
\right) $ and a generator $\hat{P}$ for $\left( \hat{T},\bar{X}\right) $ and
a measurable function $\hat{\alpha}:\bar{X}\rightarrow G$ such that $%
\left\vert \bar{P}-\hat{P}\right\vert <\varepsilon ,$%
\begin{equation*}
\int_{\bar{X}}\rho \left( \hat{\alpha}\left( \bar{x}\right) ,id_{G}\right) d%
\bar{\mu}<\varepsilon ,
\end{equation*}%
\begin{equation*}
\bar{\mu}\times \lambda \left\{ \left( \bar{x},g\right) \mid \hat{S}\left( 
\bar{x},g\right) \neq \bar{S}_{0}\left( \bar{x},g\right) \right\}
<\varepsilon ,
\end{equation*}%
and for all $n\in \mathbb{N},$%
\begin{equation*}
\left\Vert dist_{X\times G}\bigvee_{i\in \left[ n\right] }S^{-i}\left( P\vee
c\right) ,dist_{\bar{X}\times G}\bigvee_{i\in \left[ n\right] }\hat{S}%
^{-i}\left( \hat{P}\vee \hat{\alpha}c\right) \right\Vert _{\mathcal{M}}=0.
\end{equation*}%
In particular, $\left( \hat{S},\hat{T},\hat{\sigma},\bar{X}\right) $ and $%
\left( S,T,\sigma ,X\right) $ are $G-$isomorphic.
\end{theorem}

We will use the following:

\begin{lemma}
\label{copy partition}Suppose that $\left( S,T,\sigma ,X\right) $ and $%
\left( \bar{S}_{0},\bar{T}_{0},\bar{\sigma}_{0},\bar{X}\right) $ are ergodic 
$G-$ extensions on $X\times G$ and $\bar{X}\times G,$ respectively, where $%
\left( S,T,\sigma ,X\right) $ is a $G-$factor of $\left( \bar{S}_{0},\bar{T}%
_{0},\bar{\sigma}_{0},\bar{X}\right) $ via a factor map $\Phi $ of the form $%
\Phi \left( \bar{x},g\right) =\left( \phi \left( \bar{x}\right) ,g\right) .$
Suppose that $P$ is a partition of $X$ and $\bar{P}=\phi ^{-1}\left(
P\right) .$ Let $\bar{Q}$ be a partition of $\bar{X}.$ Then for all $\zeta
>0 $ and $n\in \mathbb{N}$ there is a partition $Q$ of $X$ such that%
\begin{equation}
\left\Vert dist_{X\times G}\bigvee_{i\in \left[ n\right] }S^{-i}\left( P\vee
Q\vee c\right) ,dist_{\bar{X}\times G}\bigvee_{i\in \left[ n\right] }\left( 
\bar{S}_{0}\right) ^{-i}\left( \bar{P}_{1}\vee \bar{Q}\vee c\right)
\right\Vert _{\mathcal{M}}<\zeta .  \label{dist match on tower}
\end{equation}

\begin{proof}
Choose $n_{1}>n$ and construct a Rokhlin tower $\tau $ of height $n_{1}$ for 
$S,$ measurable with respect to $X.$ Let $\bar{\tau}=\Phi ^{-1}\left( \tau
\right) .$ Choose $\zeta _{1}<\zeta $ and divide each $\left( T,P\right) $
column in $\tau $ into finitely many subcolumns on which the values of $%
\sigma $ form a set of diameter less than $\zeta _{1}.$ If $C$ is such a
column, then $\Phi ^{-1}\left( C\right) $ is a column of $\bar{\tau}$ with
the same property. We divide $\Phi ^{-1}\left( C\right) $ further into
subcolumns on each of whose levels $\bar{Q}$ is constant. Then we divide $C$
into a set of subcolumns with the same conditional distribution, and we
define $Q$ to give each the $Q-n_{1}-$name that matches the $\bar{Q}-n_{1}-$%
name of the subcolumn of $\Phi ^{-1}\left( C\right) $ that it is associated
with. If $n_{1}$ is chosen big enough, and $\zeta _{1}$ small enough, then
condition \ref{dist match on tower} is obtained.
\end{proof}

\begin{proof}
(of theorem \ref{isomorphism b}) Fix $\varepsilon >0$ and choose $\delta $
and $n$ as in theorem \ref{factor b} with respect to $\frac{\varepsilon }{2}%
. $ Suppose that $\left( S,T,\sigma ,X\right) $ and $\left( \bar{S}_{0},\bar{%
T}_{0},\bar{\sigma}_{0},\bar{X}\right) $ are ergodic $G-$extensions on $%
X\times G$ and $\bar{X}\times G,$ respectively and if $P$ is a finite
generator for $\left( T,X\right) $ and $\bar{P}$ is a partition of $\bar{X},$
such that 
\begin{equation}
\left\Vert dist_{X\times G}\bigvee_{i\in \left[ n\right] }S^{-i}\left( P\vee
c\right) ,dist_{\bar{X}\times G}\bigvee_{i\in \left[ n\right] }\bar{S}%
_{0}^{-i}\left( \bar{P}\vee c\right) \right\Vert _{\mathcal{M}}<\delta ,
\end{equation}%
By theorem \ref{factor b} there is an ergodic $G-$speedup $\bar{S}_{1}$ of $%
\bar{S}_{0}$ and a partition $\bar{P}_{1}$ of $\bar{X}$ and a measurable
function $\bar{\alpha}_{1}:\bar{X}\rightarrow G$ such that $\left\vert \bar{P%
}-\bar{P}_{1}\right\vert <\frac{\varepsilon }{2},$%
\begin{equation*}
\int_{\bar{X}}\rho \left( \bar{\alpha}_{1}\left( \bar{x}\right)
,id_{G}\right) d\bar{\mu}<\frac{\varepsilon }{2},
\end{equation*}%
\begin{equation*}
\bar{\mu}\times \lambda \left\{ \left( \bar{x},g\right) \mid \bar{S}%
_{1}\left( \bar{x},g\right) \neq \bar{S}_{0}\left( \bar{x},g\right) \right\}
<\frac{\varepsilon }{2},
\end{equation*}%
and for all $n\in \mathbb{N},$%
\begin{equation*}
\left\Vert dist_{X\times G}\bigvee_{i\in \left[ n\right] }S^{-i}\left( P\vee
c\right) ,dist_{\bar{X}\times G}\bigvee_{i\in \left[ n\right] }\left( \bar{S}%
_{1}^{\bar{\alpha}_{1}}\right) ^{-i}\left( \bar{P}_{1}\vee c\right)
\right\Vert _{\mathcal{M}}=0.
\end{equation*}%
The last condition says that the $G-$extension $\bar{S}_{1}^{\bar{\alpha}%
_{1}}$ has $S$ as a $G-$factor, via a factor map which is the identity on
the $G-$coordinate, and which has $\bar{P}_{1}$ as the preimage of $P.$
\bigskip \newline
Fix a sequence $\left\{ \varepsilon _{i}\right\} _{i=1}^{\infty }$ so that $%
\sum_{i=1}^{\infty }\varepsilon _{i}<\frac{\varepsilon }{2}$ and a sequence $%
\left\{ \bar{A}_{i}\right\} _{i=1}^{\infty }$ of sets in $\bar{X}$ that are
dense in the measure algebra of $\bar{X}$ and in which each of these sets
appears infinitely often. Choose $\delta _{1}$ and $n_{1}$ by theorem \ref%
{factor b} with respect to $\varepsilon _{1}$ and let $A_{1}\subset X$ be
chosen (using lemma \ref{copy partition}) so that, 
\begin{equation*}
\left\Vert dist_{X\times G}\bigvee_{i\in \left[ n_{1}\right] }S^{-i}\left(
P\vee \boldsymbol{1}_{A_{1}}\vee c\right) ,dist_{\bar{X}\times
G}\bigvee_{i\in \left[ n_{1}\right] }\left( \bar{S}_{1}^{\bar{\alpha}%
_{1}}\right) ^{-i}\left( \bar{P}_{1}\vee \boldsymbol{1}_{\bar{A}_{1}}\vee
c\right) \right\Vert _{\mathcal{M}}<\delta _{1}.
\end{equation*}%
Applying theorem \ref{factor b} again we get an ergodic $G-$speedup $\bar{S}%
_{2}$ of $\bar{S}_{1}$ (and hence of $\bar{S}_{0}$) and a partition $\bar{P}%
_{2}\vee \boldsymbol{1}_{\bar{A}_{1}^{\prime }}$ of $\bar{X}$ and a function 
$\bar{\alpha}_{2}:\bar{X}\rightarrow G$ such that $\left\vert \bar{P}%
_{1}\vee \boldsymbol{1}_{\bar{A}_{1}}-\bar{P}_{2}\vee \boldsymbol{1}_{\bar{A}%
_{1}^{\prime }}\right\vert <\varepsilon _{1},$%
\begin{equation*}
\int_{\bar{X}}\rho \left( \bar{\alpha}_{2}\left( \bar{x}\right)
,id_{G}\right) d\bar{\mu}<\varepsilon _{1},
\end{equation*}%
\begin{equation*}
\bar{\mu}\times \lambda \left\{ \left( \bar{x},g\right) \mid \bar{S}%
_{2}\left( \bar{x},g\right) \neq \bar{S}_{1}\left( \bar{x},g\right) \right\}
<\varepsilon _{1},
\end{equation*}%
and for all $n\in \mathbb{N},$%
\begin{equation*}
\left\Vert dist_{X\times G}\bigvee_{i\in \left[ n\right] }S^{-i}\left( P\vee 
\boldsymbol{1}_{A_{1}}\vee c\right) ,dist_{\bar{X}\times G}\bigvee_{i\in %
\left[ n\right] }\bar{S}_{2}^{\bar{\beta}_{2}}\left( \bar{P}_{2}\vee 
\boldsymbol{1}_{\bar{A}_{1}^{\prime }}\vee c\right) \right\Vert _{\mathcal{M}%
}=0,
\end{equation*}%
where $\bar{\beta}_{2}=\bar{\alpha}_{2}\bar{\alpha}_{1}.$ In other words,
the $G-$extension $\bar{S}_{2}^{\bar{\beta}_{2}}$ has $S$ as a $G-$factor,
via a factor map which is the identity on the $G-$coordinate, and which has $%
\bar{P}_{2}\vee \boldsymbol{1}_{\bar{A}_{1}^{\prime }}$ as the preimage of $%
P\vee \mathcal{A}_{1}.$ Since $P$ is a generator for $T$ we have $\bar{A}%
_{1}^{\prime }\subset \bigvee_{i=-\infty }^{\infty }\bar{T}_{2}^{-i}\left( 
\bar{P}_{2}\right) .$ Since $\left\vert \boldsymbol{1}_{\bar{A}_{1}}-%
\boldsymbol{1}_{\bar{A}_{1}^{\prime }}\right\vert <\varepsilon _{1}$ we know
that for some $m_{1}$ we have%
\begin{equation*}
\bar{A}_{1}\underset{\varepsilon _{1}}{\subset }\bigvee_{i\in \left[
-m_{1},m_{1}\right] }\bar{T}_{2}^{-i}\left( \bar{P}_{2}\right) .
\end{equation*}%
We choose $\eta _{2}>0$ so that for every transformation $\hat{T}$ of $\bar{X%
}$ and partition $\hat{P}$ of $\bar{X}$ such that 
\begin{equation}
\left\vert \hat{P}-\bar{P}_{1}^{\prime }\right\vert <\eta _{2}
\label{close partition isom}
\end{equation}%
and 
\begin{equation}
\bar{\mu}\left\{ \bar{x}\in \bar{X}\mid \hat{T}\left( \bar{x}\right) \neq 
\bar{T}_{1}\left( \bar{x}\right) \right\} <\eta _{2}
\label{small change isom}
\end{equation}%
we get 
\begin{equation*}
\bar{A}_{1}\underset{2\varepsilon _{1}}{\subset }\bigvee_{i\in \left[
-m_{1},m_{1}\right] }\hat{T}^{-i}\left( \hat{P}\right) .
\end{equation*}%
We will continue making successive speedups and partitions, making sure that
the limiting process $\left( \hat{T},\hat{P}\right) \,\ $satisfies
conditions \ref{close partition isom} and \ref{small change isom}$.$ To
proceed, we replace the numbers $\left\{ \varepsilon _{i}\right\}
_{i=2}^{\infty }$ by smaller numbers (also called $\varepsilon _{i}$) so
that $\sum_{i=2}^{\infty }\varepsilon _{i}<\eta _{2}$. We then repeat the
above argument, applying it to the partition $\bar{A}_{2}=\left\{ \bar{A}%
_{2},\bar{X}\backslash \bar{A}_{2}\right\} $ and $\varepsilon _{2}$ and the
process $\bar{S}_{2}^{\bar{\alpha}_{2}}\left( \bar{P}_{2}\vee c\right) .$
\bigskip \newline
Continuing in this way we obtain a sequence of speedups $\bar{S}_{k}$ and
partitions $\bar{P}_{k}$ and functions $\bar{\alpha}_{k}:\bar{X}\rightarrow
G $ and integers $m_{k}$ such that, for each $k$ (and writing $\bar{\beta}%
_{k}=\tprod\limits_{j=0}^{k-1}\bar{\alpha}_{k-j}$), $\left\vert \bar{P}%
_{k+1}-\bar{P}_{k}\right\vert <\varepsilon _{k},$%
\begin{equation*}
\int_{\bar{X}}\rho \left( \bar{\alpha}_{k+1}\left( \bar{x}\right)
,id_{G}\right) d\bar{\mu}<\varepsilon _{k},
\end{equation*}%
\begin{equation*}
\bar{\mu}\times \lambda \left\{ \left( \bar{x},g\right) \in \bar{X}\times
G\mid \bar{S}_{k+1}\left( \bar{x},g\right) \neq \bar{S}_{k}\left( \bar{x}%
,g\right) \right\} <\varepsilon _{k}
\end{equation*}%
for all $n$%
\begin{equation*}
\left\Vert dist_{X\times G}\bigvee_{i\in \left[ n\right] }S^{-i}\left( P\vee
c\right) ,dist_{\bar{X}\times G}\bigvee_{i\in \left[ n\right] }\bar{S}%
_{k}\left( \bar{P}_{k}\vee \bar{\beta}_{k}c\right) \right\Vert _{\mathcal{M}%
}=0,
\end{equation*}%
and%
\begin{equation*}
\bar{A}_{k}\underset{\varepsilon _{k}}{\subset }\bigvee_{i\in \left[
-m_{k},m_{k}\right] }\bar{T}_{k+1}^{-i}\left( \bar{P}_{k+1}\right) .
\end{equation*}%
Moreover, the $\varepsilon _{k}$ are chosen (by reducing all the $\left\{
\varepsilon _{i}\right\} _{i=k}^{\infty }$ at stage $k$) to guarantee that
the partitions $\bar{P}_{k}$ converge to a partition $\hat{P},$ the $\bar{S}%
_{k}$ converge to $\hat{S},$ the functions $\bar{\beta}_{k}$ converge to $%
\hat{\alpha}$ and so that $\left\vert \hat{P}-\bar{P}\right\vert
<\varepsilon ,$%
\begin{equation*}
\int_{\bar{X}}\rho \left( \hat{\alpha}\left( \bar{x}\right) ,id_{G}\right) d%
\bar{\mu}<\varepsilon ,
\end{equation*}%
\begin{equation*}
\bar{\mu}\times \lambda \left\{ \left( \bar{x},g\right) \in \bar{X}\times
G\mid \hat{S}\left( \bar{x},g\right) \neq \bar{S}_{0}\left( \bar{x},g\right)
\right\} <\varepsilon
\end{equation*}%
and for each $k$%
\begin{equation*}
\bar{A}_{k}\underset{2\varepsilon _{k}}{\subset }\bigvee_{i\in \left[
-m_{k},m_{k}\right] }\hat{T}^{-i}\left( \hat{P}\right) .
\end{equation*}%
It follows that for all $n$%
\begin{equation*}
\left\Vert dist_{X\times G}\bigvee_{i\in \left[ n\right] }S^{-i}\left( P\vee
c\right) ,dist_{\bar{X}\times G}\bigvee_{i\in \left[ n\right] }\hat{S}%
^{-i}\left( \hat{P}\vee \hat{\beta}c\right) \right\Vert _{\mathcal{M}}=0
\end{equation*}%
and that $\hat{P}$ is a generator for $\hat{T}.$ From this we conclude that
the $G-$extension $\left( \hat{S},\hat{T},\hat{\sigma},\bar{X}\right) $ is $%
G-$isomorphic to $\left( S,T,\sigma ,X\right) .$
\end{proof}
\end{lemma}

Finally, we use theorem \ref{isomorphism b}\ to prove theorem \ref%
{isomorphism a}.

\begin{proof}
(of theorem \ref{isomorphism a}) Let $\left( S,T,\sigma ,X\right) $ and $%
\left( \bar{S}_{0},\bar{T}_{0},\bar{\sigma}_{0},\bar{X}\right) $ be ergodic $%
G-$extensions on $X\times G$ and $\bar{X}\times G,$ respectively. Fix $%
\varepsilon >0.$ Choose $\delta $ and $n$ with respect to $\varepsilon $ as
theorem \ref{isomorphism b}. Let $P$ be a finite generator of $T.$ Fix $%
\zeta >0$ and $N\in \mathbb{N}$ and let $\left( x,g\right) \in \left(
X\times G\right) $ satisfy%
\begin{equation*}
\left\Vert dist_{S^{\left[ N\right] }\left( x,g\right) }\bigvee_{i\in \left[
n\right] }S^{-i}\left( P\vee c\right) ,dist_{X\times G}\bigvee_{i\in \left[ n%
\right] }S^{-i}\left( P\vee c\right) \ \right\Vert _{\mathcal{M}}<\zeta
\end{equation*}%
Let $\bar{\tau}$ be a Rokhlin tower of height $N$ for $\bar{S},$ measurable
with respect to $\bar{X},$ and define $\bar{\alpha}:\bar{X}\rightarrow G$
and $\bar{P}$ so that for each $\bar{x}$ in the base of $\bar{\tau},$ and
for all $i\in \left[ 0,N-1\right] ,$%
\begin{equation*}
\left( \bar{P}\vee \bar{\alpha}c\right) \left( \bar{S}^{i}\left( \bar{x}%
,id_{G}\right) \right) =\left( P\vee c\right) \left( S^{i}\left( x,g\right)
\right) .
\end{equation*}%
If $\zeta $ is\ chosen sufficiently small, and $N$ is sufficiently large,
then we obtain condition \ref{dist hypoth isom b} in the hypotheses of
theorem \ref{isomorphism b}. (Note that $dist_{X\times G}\bigvee_{i\in \left[
n\right] }S^{-i}\left( P\vee c\right) $ is invariant under right
multiplication in the group component, so the use of the single orbit to
define $\bar{P}$ and $\bar{\alpha}$ gives the right distribution of $n-$%
names on $\bar{X}\times G).$ The conclusion of theorem \ref{isomorphism a}
follows from the application of theorem \ref{isomorphism b}.
\end{proof}


\begin{thebibliography}{AOW}
\bibitem[AOW]{AOW} P. Arnoux, D. S. Ornstein, B. Weiss, Cutting and
stacking, interval exchanges and geometric models, \textit{Isr. J. Math}, 
\textbf{50}, , nos. 1-2, (1985), 160-168.

\bibitem[D1]{D1} R. M. Dudley, Distances of probability measures and random
variables, \textit{Ann. Math. Statistics}, 39, (1968), 1563-1572.

\bibitem[D2]{D2} R. M \ Dudley, \textit{Real Analysis and Probability},
Wadsworth \&Brooks/Cole, Pacific Grove, CA, 1989.

\bibitem[F]{F} A. Fieldsteel, Factor orbit equivalence of compact group
extensions, \textit{Isr. J. Math.}, \textbf{38}, no. 4, (1981), 289-303.

\bibitem[G]{G} M. Gerber, Factor orbit equivalence of compact group
extensions and classification of finite extensions of ergodic automorphisms, 
\textit{Isr. J. Math.}, \textbf{57}, no. 1, (1987), 28-48

\bibitem[R]{R} D. J. Rudolph, Restricted orbit equivalence\textit{,} \textit{%
Mem. AMS}, \textbf{323} (1985).

\bibitem[KR1]{KR1} J. Kammeyer, D.J. Rudolph, Restricted orbit equivalence
for ergodic $\mathbb{Z}^{d}$\ actions, I.\textit{,} \textit{Ergodic Th. Dyn.
Sys.}, \textbf{17}, no. 5, 1997, 1083--1129.

\bibitem[KR2]{KR2} J. Kammeyer, D.J. Rudolph, \textit{Restricted orbit
equivalence for actions of discrete amenable groups}, Cambridge Tracts in
Mathematics, \textbf{146}. Cambridge University Press, 2002.

\bibitem[O]{O} D. S. Ornstein, \textit{Ergodic Theory, Randomness and
Dynamical Systems}, Yale University Press, New Haven, 1970.
\end{thebibliography}
\end{document}